\documentclass[11pt]{amsart}

\usepackage{latexsym,amssymb}
\usepackage{amsmath}
\usepackage{amscd}

\usepackage{hyperref}

\usepackage{enumitem}

\input xypic
\xyoption{all}

\newcommand{\rar}{\rightarrow}
\newcommand{\lar}{\longrightarrow}

\newcommand{\llar}{-\kern-5pt-\kern-5pt\longrightarrow}
\newcommand{\lllar}{-\kern-5pt-\kern-5pt\llar}

\newtheorem{Theorem}{Theorem}[section]

\newtheorem{Corollary}[Theorem]{Corollary}
\newtheorem{Proposition}[Theorem]{Proposition}

\newtheorem{Conjecture}[Theorem]{Conjecture}

\newtheorem{Remark}[Theorem]{Remark}
\newtheorem{Example}[Theorem]{Example}

\newtheorem{Question}[Theorem]{Question}

\def\ann{\mbox{\rm ann}}

\def\codim{\mbox{\rm codim}}
\def\coker{\mbox{\rm coker}}
\def\demo{\noindent{\bf Proof. }}
\def\depth{\mbox{\rm depth }}
\def\ds{\displaystyle}

\def\Ext{\mbox{\rm Ext}}

\def\gr{\mbox{\rm gr}}
\def\reg{\mbox{\rm reg}}
\def\red{\mbox{\rm red}}

\def\Hom{\mbox{\rm Hom}}

\def\e{\mathrm{e}}

\def\m{\mathfrak{m}}

\def\QED{\hfill$\Box$}

\def\rank{\mbox{\rm rank}}

\def\Rees{\mbox{$\mathcal{R}$}}

\def\AA{{\mathbf A}}

\def\BB{{\mathbf B}}
\def\CC{{\mathbf C}}

\def\RR{{\mathbf R}}
\def\SS{{\mathbf S}}

\def\LL{{\mathbf L}}
\def\xx{{\mathbf x}}
\def\TT{{\mathbf T}}
\def\ff{{\mathbf f}}
\def\pp{{\mathbf p}}

\def\g2{{\mathbf g}}

\def\xx{{\mathbf x}}
\def\aa{{\mathbf a}}

\def\yy{{\mathbf y}}

\def\ttt{{\mathbf t}}
\def\H{{\mathrm H}}

\def\m{{\mathfrak m}}
\def\p{{\mathfrak p}}
\def\q{{\mathfrak q}}

\def\rme{{\mathrm e}}
\def\s{{\mathrm s}}
\def\red{{\mathrm r}}

\begin{document}

\title{{\sc 
The Sally modules of  ideals: a survey 
}}

\author{ Wolmer V Vasconcelos}

\maketitle



\begin{abstract}
\noindent
The Sally module of a Rees algebra $\BB$ relative to one of its Rees subalgebras $\AA$ is a construct that
can be used as a mediator 
for the trade-off of cohomological (e.g. depth) information between $\BB$ and the corresponding associated graded ring for several types 
  of filtrations. While originally devised to deal with filtrations of finite colength, here
      we treat aspects of these developments for filtrations in higher dimensions as well.

\end{abstract}



\tableofcontents

\section{Introduction}
\noindent
The aim  of  this survey is to review some of the properties of Sally modules of ideals. For a Noetherian local ring
$(\RR,\m)$ and an $\m$-primary ideal $I$, 
 these structures were 
introduced in \cite{SallyMOD} to give a quick access to some results of Judith Sally (see Section~\ref{SallyNotes}) 
connecting the depth properties 
of associated graded rings $\gr_I(\RR)$ to 
  extremal values of their multiplicity $\rme_0(I)$. They have since been used to refine this connection taking into account the roles of the higher Hilbert coefficients $\rme_1(I)$ and $\rme_2(I)$. Here we intend to visit the 
  literature to discuss a representative set of these results with emphasis on some of the techniques. The more 
  novel material
  are sought-after generalizations to a more general theory that applies to non Cohen-Macaulay rings 
  or to filtrations defined by  ideals of positive dimension. The latter are rather experimental in character.
  
  \medskip

Let $(\RR, \m)$ be Noetherian local ring of dimension $d>0$. {\em Sally modules}  
are defined in the context of Rees
algebras and some of their
finite extensions.
One of its main classes arises as follows. Let $J$ be an ideal of $\RR$ and let $\mathcal{F} = \{I_j, j\geq 0\}$, $I_0 = \RR$,  be a multiplicative  filtration of ideals. 
Consider the Rees algebras $\AA = \RR[Jt]$ and $\BB= \sum_{j\geq 0} I_jt^j= \RR\oplus \BB_{+}$. Suppose that $\BB$ is 
a finite extension of $\AA$. 
 The Sally module $S_{\BB/\AA}$ of $\BB$ relative to $\AA$ is defined by the exact sequence of finitely generated $\AA$--modules
 \begin{center}
\begin{eqnarray}\label{defnSally}
0 \rar I_1\AA \lar \BB_{+}[+1]= \bigoplus_{j\geq 1} I_{j}t^{j-1} \lar S_{\BB/\AA} = \bigoplus_{j\geq 1} I_{j+1}/I_1J^{j} \rar 0. 
\end{eqnarray}
\end{center}

\noindent The assumption is that $I_{j+1} = JI_j$ for $j\gg 1$, so that in  particular $J$ is a reduction of $I_1$.

\medskip

\noindent Some of the problems about these structures concern  the following issues:

\begin{itemize}

\item[{$\bullet$}] [Reduction number] The reduction number of $\BB$ relative to $\AA$ is defined to be  
\[ \red(\BB/\AA) = \inf \{ n \mid \BB = \sum_{j\leq n} \AA \BB_j\}.\]
We observe that $S_{\BB/\AA} = 0$ if and only if $\red(\BB/\AA) \leq 1$.  This condition is often referred as {\em minimal multiplicity}.

\item[{$\bullet$}] [Dimension] What are the possible values of $\dim S_{\BB/\AA}$? Note that for each $x\in J$, localizing at $\RR_x$ gives that $S_{\AA}(\BB)_x=0$, so that $S_{\AA}(\BB)$ is annihilated by a power of $J$ and therefore
$\dim S_{\AA}(\BB) \leq \dim \RR$.

\item[{$\bullet$}] [Multiplicity] 
What are the possible values of  multiplicity of $ S_{\BB/\AA}$ espressed in Hilbert coefficients?

\item[{$\bullet$}] [Cohen--Macaulayness] When is
 $S_{\BB/\AA}$ 
 Cohen--Macaulay?

\item[{$\bullet$}] [Regularity] There are known   relationships (\cite{Trung98}) between the  Castelnuovo regularity of the Rees algebra and that of the associated graded ring. We want to relate/extend
 comparisons to the Castelnuovo regularity of $S_{\BB/\AA}$ (if $\RR$ is not Cohen-Macaulay).

\item[{$\bullet$}] We would like to deal with these issues by relating them to some Hilbert functions associated to $\gr(\BB)$. One additional difficulty lies with the character  of the ring $\RR$ itself. This arises because
the coefficients of these functions may be positive or negative, often depending on whether $\RR$ is 
Cohen--Macaulay or not.   

\item[{$\bullet$}][Existence] Which modules are Sally modules? 

\item[{$\bullet$}][Independence] How independent of $J$ are the properties of $S_J(I)$? 

\item[{$\bullet$}][$a$-invariant] Review its role

\item[{$\bullet$}][Equations of the ideal] What are the relationships between $S_J(I)$ and the defining equations 
of $\RR[I\TT]$? What if $S_J(I) = 0$?

\item[{$\bullet$}][Non-primary ideals] Equimultiple, generically complete intersection, monomial  

\end{itemize}

\medskip

{\em Acknowledgments:} We are thankful to several authors whose work motivated the writing of this survey. We apologize for unfinished discussions and
invite readers to comment and suggestions.
We are particularly grateful to Alberto Corso, Laura Ghezzi, 
 Jooyoun Hong, Shiro Goto,  Maria Evelina Rossi,
Maria Vaz Pinto, Judy Sally,  
 Aron Simis and Giuseppe Valla.
 
\section{General Properties of  Sally Modules}



\subsection{Dimension, depth and reduction number} One of its most useful properties as a conduit of cohomological information between a Rees algebra and the associated graded ring is the following (\cite{SallyMOD}).
Throughout this section, $(\RR,\m)$ is a Noetherian local ring of dimension $d$ and $I$ is an $\m$-primary ideal. 
Whenever required will assume that $\RR$ has an infinite residue field. 

\medskip

 Cohen--Macaulay rings offer the most opportunities to exploit this notion. In the following 
 $(\RR,\m)$ is a Cohen--Macaulay local ring, $I$ is an $\m$-primary ideal,  $J$ or $Q$ denote minimal reductions 
 of $I$, and $\AA $ and $\BB$ are Rees algebras as above.

\begin{Proposition} \label{SallyJci}
Suppose $\RR$ is a Cohen--Macaulay local ring of dimension $d\geq 1$. If $J$ is an ideal generated 
by a system of parameters
 then the following hold:
\begin{enumerate}
\item If $S_{\BB/\AA}=0$, then $\gr(\BB)$ is Cohen--Macaulay.
\item If $S_{\BB/\AA}\neq 0$, then $\m \RR[JT]$ is its only associated prime.
\item If $S_{\BB/\AA}\neq 0$ then $\dim S_{\BB/\AA}= d$.
\item  $S_{\BB/\AA}$ is Cohen--Macaulay if and only if $\depth \gr(\BB)\geq d-1$.
\item If $\depth \gr(\BB) < d$ then
$\depth S_{\BB/\AA} = \depth \gr(
\BB)+1$.
\item If $S_{\BB/\AA}\neq 0$, then $\reg(S_{\BB/\AA}) = \reg(gr(\BB))=
\reg(\BB)$.
 \item Furthermore if $S_{\BB/\AA}$ is Cohen-Macaulay, then $\reg(S_{\BB/\AA}) = \reg(\BB)= \red_{\AA}(\BB) + 1$.
\end{enumerate}

\end{Proposition}

\begin{proof}
 The assertions follow from computing the cohomology (see \cite{Hu0}, \cite{JK95},
 \cite{Trung98})
 with respect to the maximal homogeneous
ideal of $\AA$ in the following exact sequences of finitely generated $\AA$--modules:

\[ 0 \rar I_1\AA^* \lar \AA^* \lar \gr_J(\RR)\otimes \RR/I_1 \simeq \RR/I_1[\TT_1, \ldots, \TT_r]^* \rar 0,\]
\[ 0 \rar \BB_{+} \lar \BB \lar \RR^* \rar 0,\]
\[ 0 \rar I_1\AA^* \lar \BB_{+}[+1] \lar S_{\BB/\AA} \rar 0,\]
\[ 0 \rar \BB_{+}[+1] \lar \BB \lar \gr(\BB) \rar 0,\]
where the modules tagged with ${}^*$ are Cohen--Macaulay. Along with the shifts, 
they allow for a measure of pivoting in all the depth calculations (and their grading in the calculation of Castelnuovo regularity).

\medskip
For the proof of (2) and (3),
in the third sequence above,
 \[ 0 \rar I_1\AA \lar \BB_{+}[+1] \lar S_{\BB/\AA} \rar 0,\]
$I_1\AA$ is a maximal Cohen--Macaulay $\AA$-module, while $\BB_{+}$ is torsion-free over $\AA$,
that is, it has the condition $S_1$ of Serre. It follows that
if $S_{\AA}(\BB)\neq 0$ then $S_{\AA}(\BB)$ has the condition $S_1$ of Serre. Therefore
$\p = \m \AA$ is its only associated, in particular $\dim S_{\AA}(\BB) = d$.

\medskip

Additionally, if  $d\geq 2$ and $\BB$ has the condition $S_2$ of Serre then so does $\BB_{+}$, and consequently $S_{\AA}(\BB)$
has the condition $S_2$ as well.
\end{proof}

One of our goals is to examine which parts of this formalism survive in case $J$ is an ideal of positive dimension, and/or $\RR$ is not Cohen--Macaulay.

\begin{Corollary}{\rm (\cite[Proposition 2.2]{GNOmmj})} Let $\AA$, $\BB$ and $\p$  be  as above.
  Then
\[s_0(S_{\AA}(\BB)) = 
\lambda(S_{\AA}(\BB)_{\p}).\]
Furthermore, if $s_0(S_{\AA}(\BB))= 1$, then $\m S_{\AA}(\BB)=0$.
\end{Corollary}

\begin{proof}
The general assertion is clear since $\p$ is the unique associated associated prime of $S_{\AA}(\BB)$ and 
$\AA/\p$ has multiplicity $1$.
On the other hand, if
$\lambda(S_{\AA}(\BB)_{\p})= 1$, 
we must have
$\p S_{\AA}(\BB)_{\p}= 0$ and therefore $\p S_{\AA}(\BB)=0$ since this module is either trivial or $\p$ would
 be its only associated prime, in which case $s_0( S_{\AA}(\BB))\geq 2$.
\end{proof}

\begin{Corollary}\label{cor1SallyJci} Let $(\RR, \m)$ be a Noetherian local ring, let $I$ be an $\m$-primary ideal and $J$ one of its minimal reductions. If for some integer $s\geq 1$, $JI^s = IJ^s$, then $\dim S_J(I)= 0$. Moreover if $\RR$ is Cohen--Macaulay then
$I^2=JI$.
\end{Corollary}

\begin{proof}  The assertion
 means that the Sally module components $I^{n+1} / J^n I = 0$ for $n \gg 0$. 
 Let $n$ be an integer of the form $n = a(s-1) + 1 \geq \red_J(I)$. 
Then
\[ I^{n+1} = J I^{a(s-1)+ 1} = J I^s I^{(a-1)(s-1)} = J^{s-1}JI^{(a-1)(s-1)+1}= \cdots = J^{a(s-1)+1}I = IJ^n.\]

The last assertion follows from Proposition~\ref{SallyJci}(3).

\medskip

Alternatively, the assumption is that $S_J(I)$ is annihilated by a power of $(J\TT)$. Since by construction $S_J(I)$ is
also annihilated by a power of $J$, the assertion follows.
\end{proof} 

\begin{Example} {\rm Let us consider some examples.
\begin{itemize}
\item $\RR = k[x,y]/x(x,y)^2$, $J = (y)$, $I = (x,y)$. Then $I^2 = (x^2, xy, y^2)\neq 
(xy, y^2) = JI $ mod $x(x,y)^2$, but $I^2J = J^2I$. 

\medskip

\item For positive integers $d > r >1$, give an example of a ring with $\dim \RR =d$ with a Sally module
with $\dim S_J(I)=r$.

\medskip

\item Adding variables to $J$ and $I$, what happens to the Sally modules? Set
$\SS'= \RR[x]$, $L = (J,x)$, $Q =(I, x)$. Then $S_Q(L) = S_J(I)[x]$.
\medskip

\item Let $\BB = k[x,y,z]$, $I = (x,y,z)^3$ and the system of parameters $(x^3, y^3, z^3)$. View 
$I$ as the  ideal of the subring $\BB_0$ of $\BB$ generated by the forms of degree $\geq 3$. Now set
$Q_0 = (x^3, y^3, z^3)\BB_0$ and $Q = Q_0 \BB$. We have the equality $S_{Q_0}(I) = S_Q(I)\neq 0$. It follows that although $\BB_0$ is 
not Cohen-Macaulay, the Sally module $S_{Q_0}(I)$ is Cohen-Macaulay.

\medskip

\item 
 Let $(\RR, \m)$ be a Cohen--Macaulay local ring and let $\p$ be an ideal of codimension $g$ such that 
$\RR_{\p}$ is not a regular local ring. If $J = (a_1, \ldots, a_g) \subset \p$ is a regular sequence, then for
the link $I = J:\p$, $I^2 = JI$ (\cite[Theorem 2.3]{CP95}). In particular $S(J;I) = 0$.

\end{itemize}

}\end{Example}

\subsection{Veronese subrings}

Let $I$ be an ideal with a reduction $J = (a_1, \ldots, a_d)$, $I^{r+1} =  JI^r$.
 For $q>0$, $J_q = (a_1^q, \ldots, a_d^q)$ is a reduction
of $I^q$, then (\cite{BF85},
\cite[Corollary 2.6]{JK94}, 
\cite[Theorem 3.10]{Planas98})
\[ \red (I^q) \leq \max \lceil 1+{\frac { \red(I) -1}{q}}, 2\rceil ,\] 
where $\lceil \quad\rceil$ is the  integral part function. In particular, for $q\gg 0$, the Sally module $S_{\BB^{(q)}/\AA^{(q)}}$ is a standard graded module. 

\begin{Proposition} If $\RR$ is Cohen-Macaulay then for $q$ larger than the postulation number of $\gr_I(\RR)$,
 $\rme_0(\AA^{(q)};\BB^{(q)}) = \rme_2(I)$.
\end{Proposition}

\medskip 

For a proof see Corollary~\ref{NorthNarita}(b). It would be interesting 
 to compare two such Sally modules with respect to their dimensions, multiplicities and depth.

\subsection{The fiber of the Sally module}
The Sally module $S_{\BB/\AA}$ is a module over  $\AA=\RR[J\TT]$. An important role here is that of the  module
\[ F_\AA(\BB) = S_{\BB/\AA}\otimes_\AA(\AA/(J\TT)) = \bigoplus_{n\geq 1}^r I_{n+1} /JI_n,\] which we call the {
its special fiber or simply the  Sally fiber}
of $S_{\BB/\AA}$. Here are some properties of $F_\AA(\BB)$:

\begin{enumerate}
\item There is a natural action of $\BB$ on $F_{\AA}(\BB)$:   For $x=b_s\TT^s\in I_s \TT^s$, 
\[x\cdot (I_{n+1}/JI_n) = xI_{n+1} + J I_{n+s}/JI_{n+s} \subset I_{n+s+1}/JI_{n+s}.\]

 \item For the filtration $I_n = I^n$,
 $ I^{r+1} \subset JI^r$, in particular $r=\red_{\AA}(\BB)\leq \lambda(F_\AA(\BB))+1$.  
 
 \item In the case of equality $I^{n+1}/JI^ n\simeq \RR/\m$ for 
 all $1\leq n\leq r-1$. 
 
 \item We conjecture that in this case all $I^n$ are Ratliff-Rush closed! Check in \cite{CPP98}.

 \item This can be further refined by considering the action of $I_s\TT^s$ on $F_\AA(\BB) \otimes \RR/\m = \overline{F}$ and benefit
 from Nakayama's Lemma: 
 $\red_{\AA}(\BB) \leq \lambda(\overline{F})+1$.

\item Let $I^n$ be an $I$-adic filtration and $\widetilde{I^n}$ its Ratliff-Rush closure. There is a natural mapping of Sally modules
\[ I^n/I J^{n-1} \mapsto \widetilde{I^n}/\widetilde{I}J^{n-1}.\] If $I = \widetilde{I}$
this mapping is an embedding that is an
 isomorphism for $n \gg 0$.

 \item In this case the Hilbert polynomials agree so that in particular $
 s_0(S_{\BB/\AA}) = s_0(S_{\widetilde{\BB}/\AA})$.
 
 \medskip
 
 


 
 \end{enumerate}

\section{Dimension, Multiplicity and Euler Characteristics}
\noindent
We begin to develop the basic metrics to examine Sally modules with a focus on the determination of multiplicities. 

\subsection{The  Hilbert series and dimension} A basic calculation refines  the relationship between
the multiplicity and the dimension
of the Sally module.

\begin{Proposition}
 Let $(\RR, \m)$ be a Noetherian local ring of dimension $d$, $I$ an $\m$-primary ideal and $Q$ one of 
its minimal reductions. 
Then $\dim S_Q(I) = d$ if and only if $s_Q(I) \neq 
0$.
\end{Proposition}

More generally we have observed that $\dim S_Q(I)\leq d$. This means that its Hilbert polynomial
\[ \begin{array}{rcl}
{\ds \lambda(I^{n+1}/I Q^{n})} &=  & {\ds \sum_{0}^{d-1}
(-1)^i s_ i {{n+i-1}\choose{d-i-1}}},
\end{array}
\] 
where we denote $s_Q(I)= s_0$. Finally observe
\begin{Proposition} \label{HilbdimSally}
 If $ S_ Q(I) \neq 0$ then
\[
\dim  S_Q(I) = d-\min\{i \mid s_i\neq 0\}, \ \mbox{\rm or $0$} .\]
\end{Proposition}
 
\begin{Proposition}\label{Sallydsequence}
 Let $(\RR, \m)$ be a Noetherian local ring of dimension $d$, $I$ an $\m$-primary ideal and $Q$ one of 
its minimal reductions. Suppose that the corresponding Sally module   $ S_{Q}(I)$ has dimension $d$. Then its multiplicity satisfies
\[ \s_{Q}(I) \leq \rme_1(I) -\rme_1(Q) - \rme_0(I) + \lambda(\RR/I ).\] 
 \end{Proposition}  

\begin{proof}
Let $S_{Q}(I)= \bigoplus_{n\geq 1} I^{n+1}/I Q^{n}$ be the Sally module of $I$ relative to $Q$.  From the exact sequences
\[ \begin{array}{rcccccccc}
 0  & \rar & I^{n+1}/IQ^{n} & \lar & \RR/IQ^{n} & \lar & \RR/I^{n+1} & \rar & 0 \\  &&&&&&&&\\
 0  & \rar & Q^{n}/IQ^{n} & \lar & \RR/IQ^{n} & \lar & \RR/Q^{n} & \rar & 0 \\
\end{array}\]
we obtain that
\begin{eqnarray}\label{Hilbpoly}
  \lambda(I^{n+1}/I Q^{n} )= - \lambda(\RR/I^{n+1}) + \lambda(\RR/Q^{n}) +\lambda(Q^{n}/I Q^{n}).
  \end{eqnarray}
For $n \gg 0$, we have
\[ \begin{array}{rcl}
{\ds  \lambda(I^{n+1}/I Q^{n}) }&=  & {\ds \s_{Q}(I){{n+d-1}\choose{d-1}}+{\rm lower \ degree \ terms}} \\ && \\
\lambda(\RR/I^{n+1}) &= & {\ds \rme_0(I){{n+d}\choose{d}}-\rme_1(I){{n+d-1}\choose{d-1}}+{\rm lower \ degree \ terms}.}
\end{array}
\] Since $\rme_0(Q)=\rme_0(I)$, for $n \gg 0$ we have
\[ \begin{array}{rcl}
{\ds \lambda(\RR/Q^{n}) } &=& {\ds \rme_0(Q){{n+d-1}\choose{d}}-\rme_1(Q){{n+d-2}\choose{d-1}}+{\rm lower \ degree \ terms}} \\ && \\
 &=& {\ds \rme_{0}(Q) \left( {{n+d}\choose{d}} - {{n+d-1}\choose{d-1}} \right) - \rme_{1}(Q) \left( {{n+d-1}\choose{d-1}} - {{n+d-2}\choose{d-2}}  \right) +  \cdots    } \\  && \\
 &=& {\ds \rme_{0}(I) {{n+d}\choose{d}} - \rme_{0}(I) {{n+d-1}\choose{d-1}}  - \rme_{1}(Q) {{n+d-1}\choose{d-1}} +{\rm lower \ degree \ terms}  } \\
\end{array} \]
Let $\overline{G}=\gr_Q(\RR)\otimes \RR/I$. Then for $n \gg 0$ we have
\[ \lambda(Q^{n}/I Q^{n})=\rme_0(\overline{G}){{n+d-1}\choose{d-1}}+{\rm lower \ degree \ terms}. \]
It follows that 
\[ \s_{Q}(I)=\rme_1(I)  - \rme_0(I) -\rme_1(Q) + \rme_0(\overline{G}).\]
But $\overline{G}$ is a homomorphic image of a polynomial ring $\RR/I[\TT_1,\ldots, \TT_d]$, so that
$\rme_0(\overline{G}) \leq \lambda(\RR/I)$, 
 and we obtain the desired upper bound for $\s_{Q}(I)$. 
\end{proof}


In a minor refinement, write (\ref{Hilbpoly}) as
\begin{eqnarray}\label{Hilbpoly2}
  \lambda(I^{n+1}/I Q^{n} )= - \lambda(\RR/I^{n+1}) + \lambda(\RR/Q^{n+1})
 -\lambda(Q^n/Q^{n+1}) 
 +\lambda(Q^{n}/I Q^{n}).
  \end{eqnarray}
  
We now match the Hilbert coefficients:

\begin{eqnarray*}
0 & = & \rme_0(I) - \rme_0(Q) \\
s_Q(I) = s_0 & = & \rme_1(I) - \rme_1(Q) + \rme_0(Q) + \rme_0(\overline{G}) \\
& \vdots & \\
 s_i & = & \rme_{i+1}(I) - \rme_{i+1}(Q) + \rme_{i}(Q) + \rme_i(\overline{G}), \ i\geq 1. \\
\end{eqnarray*}

To apply Proposition~\ref{HilbdimSally} is simpler when $\rme_i(\overline{G})=0$ for $i\geq 1$--and we treat two cases here. Another simplification is when the values of $\rme_i(Q)$, for $i\geq 1$, are 
independent of $I$ (Buchsbaum rings, to check).

\medskip
  
  This highlights the fact that, unlike the Cohen--Macaulay case, several properties of $S_Q(I)$ [dimension, multiplicity] may depend on $Q$. Meanwhile we have


\begin{Corollary} Let $(\RR, \m)$ be a Cohen-Macaulay local ring of dimension $d\geq 1$, let $I$ be an $\m$-primary ideal and let $Q$ be one of
its minimal reductions generated by a regular sequence. Then
\begin{enumerate}

\item[{\rm (1)}] The Hilbert function of $S_Q(I)$ is independent of $Q$.

\item[{\rm (2)}] If $d=1$, $\red_Q(I)$ is independent of $Q$. More generally, the assertion holds if $\depth \gr_I(\RR) \geq d-1$.  More precisely, 
if $ {h(t)}\over {(1-t)^d}$ is the Hilbert series of $S_Q(I)$, then $\red_Q(I) = \deg h(t)$.

\end{enumerate}
\end{Corollary}

\begin{proof} (1) follows directly from the formula
\[\lambda(I^{n+1}/I Q^{n} )= - \lambda(\RR/I^{n+1}) + \lambda(\RR/Q^{n}) +\lambda(Q^{n}/I Q^{n}),\]
 as the last two terms are respectively
 \begin{eqnarray*}
 \lambda(\RR/Q^n) & = & \rme_0(I) {{d+n-1}\choose{d}}, \\
 \lambda(Q^n/IQ^n) & = & \lambda(\RR/I) {{d+n-1}\choose{d-1}}. \\
 \end{eqnarray*}
 
 (2) Let us consider a general observation first.
 From the definition of $S_ Q(I)$,
 \[ 0\rar I \RR[ Q\TT] \lar I \RR[I \TT] \lar S_Q(I) \rar 0,\]
 we see that $S_ Q(I)$ is a module over $\RR[Q\TT]$ and as such has the condition $S_1$ of Serre.  Considering the
 natural presentation of the 
 polynomial ring $\AA = \RR[t_1, \ldots, t_d] \rar \RR[Q\TT]$, we view $S_Q(I)$ as a finite $\AA$-module.   
 
 Of course the action of the variables depend on $Q$ as does $L$ the annihilator of $S_Q(I)$ as an $\RR$-module. Nevertheless as
 an $\AA/L \AA$-module $S_Q(I)$ keeps the condition $S_1$ and therefore
 there is a form of degree $1$ that induces
 an injective homomorphism on $S_Q(I)$. Consequently its Hilbert function is monotonically increasing.
 
 Let $d=1$ now and suppose $S_Q(I) \neq 0$ [otherwise there is no difficulty]: With $\AA= \RR[t]$ the action of 
 \[ t: S_{n-1} \mapsto S_{n} \]
 is surjective if $I^{n+1} = QI^n$ thus $n\geq \red_ Q(I)$. Since this means that
  $\lambda(I^n/IQ^{n-1}) =	
 \lambda(I^{n+1}/IQ^{n})$, a condition that is independent of $Q$, This proves the assertion.
 
\medskip

SimiIar observations can be made if $\depth \gr_I(\RR) \geq d-1$ since this means that $S_Q(I)$ is Cohen-Macaulay 
for each $Q$ and therefore the reduction number $\red_Q(I)$ is given by the degree of the $h$-polynomial minus one.  
   
 \end{proof}

A general result that also addresses the case when $S_Q(I)$ is Cohen-Macaulay is \cite[Theorem 1.2]{Trung87}.

\begin{Corollary} If $\RR$ is Cohen-Macaulay then
\[ s_{Q}(I) = \rme_1(I) - \rme_0(I) +\lambda(\RR/I),\]
in particular the  dimension, the depth  and the multiplicity of $S_Q(I)$ are independent of $Q$.
\end{Corollary}

\begin{proof}
The assertions about the dimension and the depth follow from Proposition~\ref{SallyJci}.
\end{proof}

\begin{Remark}\label{s0gen}{\rm
Suppose $\mathcal{F} = \{I_n\}$ is a filtration so that the corresponding Rees algebra $\BB$ is finite over 
$\AA = \RR[Q\TT]$. 
Let $S_{Q}(\BB)= \bigoplus_{n\geq 1} I_{n+1}/I_1 Q^{n}$ be the associated  Sally module.
The same calculation yield the formulas
\[ \s_{Q}(\mathcal{F}) \leq \rme_1(\mathcal{F}) -\rme_1(Q) - \rme_0(Q) + \lambda(\RR/I_1 ),\]
and if $\RR$ is Cohen--Macaulay
\[ \s_{Q}(\mathcal{F}) = \rme_1(\mathcal{F})  - \rme_0(Q) + \lambda(\RR/I_1 ).\]

}\end{Remark}

Let us consider a mild generalization of these formulas.

\begin{Corollary}
Let $(\RR, \m)$ be a Noetherian local ring of dimension $d$, $I$ an $\m$-primary ideal and $Q$ one of  its minimal reductions.
Suppose that $\dim S_Q(I)=d$. Let $\varphi$ is the matrix of syzygies of $Q$.
If $Q$ is generated by a d-sequence and $I_1(\varphi)\subset I$, then
\[ \s_{Q}(I) = \rme_1(I) -\rme_1(Q) - \rme_0(I) + \lambda(\RR/I ).\] 
\end{Corollary}

\begin{proof}
Let $\overline{G}=\gr_Q(\RR)\otimes \RR/I$. By the proof of Theorem~\ref{Sallydsequence}, we have
\[ \s_{Q}(I)=\rme_1(I)  - \rme_0(I) -\rme_1(Q) + \rme_0(\overline{G}).\]
Set $\SS=\RR[\TT_{1}, \ldots, \TT_{d}]$. Since $Q$ is generated by a $d$-sequence, the approximation complex  
 \[ 0\rar \H_d(Q)\otimes \SS[-d] \rar \cdots \rar \H_1(Q)\otimes \SS[-1] \rar  \H_0(Q)\otimes \SS\rar \gr_Q(\RR) \rar 0  \]
is acyclic (\cite[Theorem 5.6]{HSV3}).
By tensoring this complex by $\RR/I$, we get the exact 
sequence
\[  \H_1(Q) \otimes \RR/I \otimes \SS[-1] \stackrel{\phi}{\lar} \RR/I \otimes \SS
 \lar \overline{G} \rar 0.\]
 Since $I_1(\varphi)\subset I$,  the mapping $\phi$ is trivial so that $\rme_0(\overline{G})=\lambda(\RR/I)$.
\end{proof}

\begin{Remark}{\rm 
One instance when $I_1(\varphi) \subset I$ occurs  if $\RR$ is unmixed and $I$ is integrally closed, according
to \cite[Proposition 3.1]{Rat74}.
}\end{Remark}

\begin{Corollary} Suppose $(\RR,\m)$ is an unmixed
 Buchsbaum local ring of dimension $d$. Let $I$ be an $\m$-primary integrally closed  ideal and $Q$ is one 
 of its minimal reductions. If $\dim S_Q(I) = d$ then the dimension and multiplicity of the Sally
 modules of $I$ are independent of the chosen minimal reduction.
 \end{Corollary}
 
 \begin{proof} By assumption $S_Q(I) \neq 0$.  Since
    $\rme_1(Q)$ is independent of $Q$ for Buchsbaum
 rings (\cite{GO1}), $S_ Q(I)$ is   unchanged across all minimal reductions of $I$. 
In particular $\dim  S_Q(I)$ is also unchanged.
 \end{proof}
 

The following classical inequalities for the Hilbert coefficients of primary ideals
arises from this formula (see \cite{CPP98} for a detailed discussion).

\begin{Corollary} \label{NorthNarita}
Let $(\RR, \m)$ be a Cohen--Macaulay local ring of dimension $d$ and let $I$ be a 
$\m$-primary ideal. Then 
\[\begin{array}{ll}
 \mbox{\rm (a) Northcott inequality \cite{Northcott1}} & \rme_1(I) -\rme_0(I) + \lambda(\RR/I)\geq 0.\\
\mbox{\rm (b) Narita inequality \cite{Narita}} & \rme_2(I) \geq 0.
\end{array}\]
\end{Corollary} 

\begin{proof} For (a), one invokes the fact that for $\RR$ Cohen-Macaulay, $\rme_1(I) -\rme_0(I) + \lambda(\RR/I)$ is the 
multiplicity of $S_Q(I)$. For (b), we may assume that $d=2$ and  using a trick of J. Lipman to evaluate this multiplicity on
 $I^q$ for $q$ larger than the postulation number of  the Hilbert function of $\gr_I(\RR)$.

We first consider the case for $d=2$ and without great loss of
generality take $I_n=I^n.$ Let
\[ \lambda(\RR/I^n)= \rme_0{{n+1}\choose{2}} -\rme_1{{n}\choose{1}} + \rme_2,
\quad n \gg 0\]
be the Hilbert polynomial of $I$. Let $q$ be an integer greater than
the postulation number for this function. We have
\begin{eqnarray*} \lambda(\RR/I^{qn}) = \lambda(\RR/(I^q)^n)&=& f_0{{n+1}\choose{2}} -f_1{{n}\choose{1}} +
f_2 \\ &=& \rme_0{{qn+1}\choose{2}} -\rme_1{{qn}\choose{1}} + \rme_2,
\quad n \gg 0.\end{eqnarray*}
Comparing these two polynomials in $n$, gives
\begin{eqnarray*}
f_0 & = & q^2\rme_0 \\
f_1 & = & q\rme_1 + \frac{1}{2}q^2\rme_0-\frac{1}{2}q\rme_0 \\
f_2 &=& \rme_2.
\end{eqnarray*}
Since the multiplicity $s_0(I^q)$ of the
Sally module of $I^q$ is given by
\[ s_0(I^q) = f_1-f_0+ \lambda(\RR/I^q),\] which by the equalities
above yields
$\rme_2(I) = s_0(I^q)$, the multiplicity of the Sally module of $I^q$.
Thus $\rme_2(I)$ is nonnegative and vanishes if
 and only if the reduction number of $I^{q}$ is less than or equal to $1$.
In dimension $d> 2$, we reduce $\RR$ and $I$ modulo a superficial
element.
\end{proof}

\subsubsection{Maximal ideals} The dimension of Sally modules over non-Cohen-Macaulay local rings is
hard to ascertain. We state a result of A. Corso (\cite{C09}) for maximal ideals.

\begin{Theorem}{\rm \cite[Theorem 2.1]{C09}} \label{Alberto21}
Let $(\RR, \m)$ be a local Noetherian ring of dimension $d>0$ with infinite residue field, and let $Q$
 be a minimal reduction of $\m$. 
\begin{enumerate}
\item[{\rm (a)}] The Sally module $S_{Q}(\m)$ has dimension $d$ if and only if $\rme_1(\m) -\rme_0(\m) - \rme_1(Q) + 1$ is strictly positive. Otherwise, its dimension is $d - \inf\{ i \mid 1 \leq i \leq d-1, \rme_{i+1}(\m) -\rme_{i}(Q) - \rme_{i+1}(Q) \neq 0 \}$ or $0$.

\item[{\rm (b)}] If $\RR$ is a Buchsbaum ring, the Sally module $S_{Q}(\m)$ has dimension $d$ or $0$.
\end{enumerate}
\end{Theorem}

\begin{Remark}
{\rm For $\m$-primary ideals $I \neq \m$, a similar description   is not known, even for Buchsbaum
 rings.}
 \end{Remark}

\subsection{Cohen--Macaulayness of the Sally module}

As a motivation, it is worthwhile to point out one useful aspect of these modules.
Under the conditions of (\ref{SallyJci}),
if $I_1$ is $\m$--primary there is a numerical test for the Cohen-Macaulayness of $S_{\BB/\AA}$ 
(\cite{Huc96}; see also \cite{acm}). It does not always require that $\RR$ be Cohen-Macaulay. 

\begin{Theorem}[Huckaba test] \label{Huctest} If $\RR$ is Cohen-Macaulay,
the Sally module $S_{\BB/\AA}$ is Cohen--Macaulay if and only if
\[ \e_1(\gr(\BB)) = \sum_{j\geq 1} \lambda(I_j/QI_{j-1}).\]
\end{Theorem}

\begin{proof}
Its brief proof, as in \cite{acm}, goes as follows. First
 note that $S_{\BB/\AA}$ is an $\AA$--module annihilated by a power of $Q$, from which
it follows    that  the generators of $Q\TT$ is a system of
parameters.
In the case of an $\m$--primary ideal $Q$, the interpretation of this formula is
\[ \e_0(S_{\BB/\AA})= \lambda(S_{\BB/\AA} \otimes \RR[Q\TT]/(Q\TT)) 
.\] In other words,
$\chi_1(Q\TT; S_{\BB/\AA}) =0$, and therefore $S_{\BB/\AA}$ is Cohen--Macaulay by Serre's Theorem
(\cite[Theorem 4.7.10]{BH}) on Euler's characteristics.
\end{proof}

Here is an important case that does not use this test.

\begin{Theorem}  {\cite[Corollary 3.8]{Rossi00a}}.
Let
$(\RR, \m)$ be  a Cohen--Macaulay local ring of dimension $d$.
and let
$I$ be an $\m$-primary ideal. If
\[\rme_0(I)= \lambda(I/I^2) + (1-d)\lambda(\RR/I) + 1\] then
 $\depth (G) \geq d-1$. In particular  all the Sally modules of $I$  are Cohen--Macaulay.
 \end{Theorem}

\begin{Remark}{\rm The {\em equations of the ideal $I$} (\cite{WV91}) are the the ideal of relations of
a presentation
\[ 0 \rar \mathbf{L} \lar \CC = \RR[\TT_1, \ldots, \TT_n] \stackrel{\varphi}{\lar} \RR[I\TT] \rar 0.\] 
$\mathbf{L}=\bigoplus_{j\geq 1} L_j $ is a graded ideal whose properties are independent of the presentation.
For instance, the {\em relation type} of $I$, $\mbox{\rm rel}(I)$, is the maximum  degree of a minimal set of homogeneous generators of 
$\mathbf{L}$. If $\RR$ is a  Cohen--Macaulay local ring and $I$ in $\m$-primary, then the Sally module
$S_J(I)$ impacts the relation type in a number of ways, for instance if $S_J(I)$ is Cohen--Macaulay then
$\mbox{\rm rel}(I) \leq \red_J(I) + 1$, according to   
\cite[Theorem 1.2]{Trung87}.

\medskip
Several questions can be raised if $I$ is not $\m$-primary. We will attempt to do this later.
}\end{Remark}

 
\subsubsection{Dimension one} Let $(\RR, \m)$
be a Cohen-Macaulay local ring of dimension one. By
Proposition~\ref{SallyJci}(iv), for each $\m$-primary ideal $I$ the Sally module $S=S_Q(I)$
is Cohen-Macaulay, and therefore by Theorem~\ref{Huctest}
\[ \rme_1(I) = \sum_{j\geq 1} \lambda(I_j/QI_{j-1}).\]

\begin{Proposition}\label{SallyofC} Let $ (\RR, \m)$ be a Cohen--Macaulay local ring of dimension
$1$ and let $I$ be an $\m$-primary ideal. If $Q=(a)$ is 
a minimal reduction of $I$ then the corresponding Sally module $S=S_Q(I)$ has the following
properties. Setting $L = \ann(I/Q)$,
\begin{itemize}
\item[{\rm (i)}] $S$ is Cohen-Macaulay.
\medskip

\item[{\rm (ii)}] If $s=\red(I) \geq 2$,  then \[\e_1(I) = \lambda(\RR/L)  + \sum_{j=1}^{s-1} \lambda(I^{j+1}/aI^j).\]


\item[{\rm (iii)}] If $\nu(I)=2$, then \[\red(I) \geq {\frac{\e_1(I) }{\lambda(\RR/
L)}}.\]

\item[{\rm (iv)}] If $\nu(I) = \red(I) = 2$, then $\e_1(I) \leq  2\cdot 
\lambda(\RR/L)$. 

\end{itemize}
\end{Proposition}
 
\begin{proof} (i), (ii): Since $\depth \gr_I(\RR) \geq d-1$, $S$ is Cohen-Macaulay.

\medskip

(iii): If $I = (a, b)$, $I^{j+1}/QI^{j}$ is cyclic and annihilated by $L$, so
\[\e_1(I) \leq \red(I) \cdot \lambda(\RR/L).\] 

\medskip

(iv): is a special case of (iii).

\end{proof}

\subsection{Change of rings for Sally modules} Let $(\RR,\m)$ be a Noetherian local ring of dimension $d\geq 1$ and let 
$\ff: \RR \rar \SS$ be a finite morphism. Suppose $\SS$ is a Cohen--Macaulay ring locally of dimension $d$.
 For each $\m$-primary ideal $I$ with a 
minimal reduction $Q$, for each maximal ideal $\mathfrak{M}$ of $\SS$
 the ideal $I\SS_{\mathfrak{M}}$ is ${\mathfrak M}\SS_{\mathfrak{M}}$-primary and $Q\SS_{\mathfrak{M}}$ is a minimal reduction.
We define the Sally module of $I\SS$ relative to $Q\SS$ in the same manner,
\[ S_{Q\SS}(I\SS) = \bigoplus_{n\geq 1} I^{n+1}\SS/I Q^{n}\SS.\]
$S_{Q\SS}(I\SS)$ is a finitely generated graded $\gr_Q(\RR)$-module. 
Its localization at $\mathfrak{M}$ gives $S_{Q\SS_{\mathfrak{M}}}(I\SS_{\mathfrak{M}})$.  

\medskip
 Consider the natural mapping 
\[ 
\varphi_{\ff}: \SS \otimes_{\RR} S_Q(I)  
=  \bigoplus_{n\geq 1} I^{n+1}/IQ^{n} \otimes_{\RR}\SS \mapsto
 S_{Q\SS}(I\SS) =
  \bigoplus_{n\geq 1} I^{n+1}\SS/IQ^{n}\SS.\] 
$\varphi_{\ff}$ is a graded surjection from $\SS \otimes_{\RR} S_Q(I)$ onto $S_{Q\SS}(I\SS)$, which combined with a presentation
$\RR^b \rar \SS \rar 0$, $b = \nu(\SS)$, gives rise to a homogeneous  surjection of graded modules
\[
 [S_Q(I)]^b \mapsto S_{Q\SS}(I\SS).
\]
The following takes information from this construction into Proposition~\ref{SallyJci}. It is  a factor in our estimation of several
reduction number calculations, according to \cite{red}.

\begin{Theorem}[Change of Rings Theorem]\label{changeSally} Let $\RR, \SS, I, Q$ be as above. If $\SS$ is
Cohen-Macaulay
then
\begin{enumerate}
\item[{\rm (1)}]
 If $\dim S_Q(I) < d$, then $S_{Q\SS}(I\SS)= 0$.
\item[{\rm (2)}] If $s_0(I\SS) \neq 0$, then $\dim S_Q(I) = d$.
\item[{\rm (3)}] If $\dim S_Q(I) = d$, then $s_0(I\SS) \leq b\cdot s_0(I)$.
\end{enumerate}
\end{Theorem}

We note that $s_0(I\SS)$ is a multiplicity relative to $\RR$. It is a positive summation of  local multiplicities, so each of these is also bounded.

\begin{proof} All the assertions follow from the surjection ${ [S_Q(I)]^b \rar S_{Q\SS}(I\SS)}$ and the vanishing property of Sally modules over Cohen-Macaulay rings 
(Proposition~\ref{SallyJci}.2).
\end{proof}

\begin{Corollary}
$ s_0(I\SS) \leq b\cdot s_0(I)$.
\end{Corollary}

\begin{Remark}{\rm It would be good to prove that the condition  (\ref{changeSally}.2) is actually an equivalence.}
\end{Remark}

\begin{Question}{\rm
Let $\SS = k[x_1, \ldots, x_d]$ and let $G$ be a finite group of $k$-automorphisms of $\SS$. Set $\RR = \SS^G$. 

[{\bf True or false?}] For any ideal $I$ of $\RR$ of finite colength $S_Q(I)$ is [locally] either $0$ or has dimension $d$. 

\medskip

If $Q$ is a minimal reduction of $I$, $Q\SS$ is a minimal reduction of $I\SS$. Suppose $I^2\SS = QI\SS$.

}\end{Question}

\subsubsection*{Extra remarks on change of rings} Suppose  $(\RR,  d, \m, I, Q)$ 
is as above and $\varphi : \RR \rar \SS$ is an injective finite morphism.

\begin{itemize}
\item We would like to prove $\dim S_Q(I) = \dim S_{Q\SS}(I\SS)$. Set $\CC = \SS[Q\SS \TT]$ and consider 
the sequence of natural surjective homomophisms of $\BB$ modules
\[ 
\SS\otimes_{\RR} S_Q(I) = \SS\otimes_{\RR}\BB \otimes_{\BB} S_Q(I)
\rar
\CC\otimes_{\BB} S_Q(I) \rar
 S_{Q\SS}(I\SS).
\]

\item {\bf Claim 1.} $\dim \CC\otimes_{\BB} S_Q(I) = \dim S_{Q}(I)$: Let $P$ be an associated prime
ideal of $S_Q(I)$, as a $\BB$-module,  with $\dim \BB/P = \dim S_Q(I)$. Localize at $P$. Let $L \subset
P\BB_P$ be a system of parameters contained in the annihilator of $S_Q(I)_P$. It follows easily
that $ \CC_P/L\CC_P \otimes_{\BB}  S_Q(I)_P\neq 0$, and therefore $\CC \otimes_{\BB} S_Q(I)$  has dimension
$\dim \BB/P$, as desired.

\medskip

\item {\bf Claim 2.} We will argue that $\dim S_Q(I)= \dim S_{Q\SS}(I\SS)$, as $\BB$-modules.

\end{itemize}

\subsection{Special fiber and reduction number}
Following \cite{CPV6} we  study relationships between $\red(I)$ and the multiplicity $f_0(I)$ of the special fiber
${F}(I)= \RR[I\TT] \otimes \RR/\m$ of the Rees algebra of $I$. 
Let $(\RR,\m)$ be a Cohen--Macaulay local ring of dimension $d\geq 1$, and let $I$ be an $\m$-primary ideal,
 let $Q$ be a minimal reduction of $I$ and let $S_Q(I)$ denote the Sally module of $I$ relative to $Q$. Suppose
 $I$ is minimally generated by $s$ elements. We can write $ I = (Q, b_1, \ldots, b_{s-d})$.
 Consider the exact sequence introduced in \cite[Proof of 2.1]{CPV6} 
 \begin{eqnarray}\label{cpv6}
  \RR[Q\TT]  \oplus \RR[Q\TT][-1]^{s-d} \stackrel{\varphi}{\lar}  \RR[I\TT] \lar  S_Q(I)[-1]
 \rar 0,
 \end{eqnarray}
where $\varphi$ is the map defined by $\varphi (f,a_1, \ldots, a_{s-d}) = f + \sum_ja_ib_j \TT$. Tensoring this sequence with $\RR/\m$ yields the 
  exact sequence
\begin{eqnarray*}\label{sequence}
F(Q) \oplus F(Q)[-1]^{s-d} \stackrel{\overline{\varphi}}{\lar} {F}(I) \lar  S_Q(I)[-1]\otimes \RR/\m
 \rar 0.\end{eqnarray*}
 
 We make some observations about the relationship between $s_0(I)$ and $\red_Q(I)$ arising from this complex
  in special cases:
 
 \medskip
 
 \begin{itemize}
 
 \item[{$\bullet$}]  The exact sequence gives
 \[ f_0(I) \leq \deg \overline{S_Q(I)} + \nu(I) -d+1\leq  \deg S_Q(I) + \nu(I) -d+1.\]

 \medskip

 \item[{$\bullet$}] If $\m S_Q(I) = 0$, $f_0(I) \leq s_0(I) + \nu(I)-d+1$.  
  If $\m S_Q(I) \neq 0$, $\dim (\m S_Q(I))=d$ since $\m \RR[Q\TT]$ is the only associated prime of
  $S_Q(I)$. 
 Therefore $\deg \overline{S_Q(I)} < \deg S_Q(I)$ and  thus $f_0(I) \leq s_0(I) + \nu(I) -d $.
 
 \medskip

 \item[{$\bullet$}] Suppose $I$ is an almost complete intersection. If $\m S_Q(I)=0$, we are in a case  similar to that treated in
 \cite[Corollary~3.8]{red}. Then $F(I) $ is a hypersurface ring and
  $\red_Q(I) + 1=f_0(I) = s_0(I) + 2$.
 
 \medskip
 
\end{itemize} 

 To benefit from these bounds, we need information about $F(I)$. For example, if $I$ and $Q$ are given by forms of the same degree, $F(I)$ is an integral domain. Such conditions could be used in generalized versions of the classical Cayley-Hamilton theorem:
 
 \begin{Theorem}\label{CHTheorem} Let $(\RR,\m)$ be a Noetherian local ring, $I$ an ideal and $Q$ one of its minimal reductions. Then
 \begin{enumerate}
 \item[{\rm (1)}]
 $\red_Q(I) \leq \nu_{F(Q)}(F(I))-1$.
 \item[{\rm (2)}] 
 If $F(I)$ satisfies the condition $S_1$ of Serre,
$\red_Q(I) \leq f_0(I)-1$.
 \end{enumerate}
 \end{Theorem}

\begin{proof} (1) is the standard
 Cayley-Hamilton theorem. The condition in (2) means  that $F(I)$ is a torsionfree 
graded module of rank $f_0(I)$
over the polynomial ring $F(Q)$. One then invokes \cite[Proposition 9]{rn1}. 
\end{proof} 

\subsection{Multiplicity and reduction number}
The following result of M. E. Rossi (\cite{Rossi00}) is the main motivation of our discussion here.

\begin{Theorem}\label{Rossibound}{\rm (\cite[Corollary 1.5]{Rossi00})}
 If $(\RR,\m)$ is a Cohen-Macaulay local ring of dimension at most $2$, then for any $\m$-primary ideal $I$ with a minimal reduction $Q$,
 \[ \red_{Q}(I) \leq \rme_1(I) - \rme_0(I)+ \lambda(\RR/I) + 1.\]
\end{Theorem}

Given that  
$  \rme_1(I) - \rme_0(I)+ \lambda(\RR/I)$ is the multiplicity $s_0(Q;I)$ of the Sally module $S_Q(I)$, we still consider:

 \begin{enumerate}
 
 \item If $\RR$ is Cohen-Macaulay, \cite{Rossi00} raised the question on whether in all dimensions
 \[ \red_Q(I) \leq s_0(Q;I) + 1,\]
a fact observed in numerous cases. 
 \medskip

 \item If $\RR$ is not Cohen-Macaulay the expression for the multiplicity of $S_Q(I)$ is different [as seen in some cases above]. Nevertheless we do not have a failure for this bound for $\red_Q(I)$.
 \medskip

 \item If $\RR$ is not Cohen-Macaulay this formula may require the addition, or multiplication, by  a constant
 that expresses some defficiency $d(\RR)$  of  $\RR$  being non-Cohen-Macaulay, such as
 \begin{itemize}
 \item[{$\bullet$}]
 $ \red_Q(I) \leq d(\RR) \cdot s_0(Q;I)$,
 or
 \item[{$\bullet$}] 
 $ \red_Q(I) \leq  s_0(Q;I)+d(\RR)$.
 \end{itemize}
 
 \medskip

 

\end{enumerate}

\subsection{Sally questions}\label{SallyNotes}


Let $(\RR,\m)$ be a Cohen--Macaulay local ring of dimension $d$. Denote the embedding dimension of $\RR$ by
  $ n = \nu(\m)$ and $\rme=\rme_0(\m)$ its multiplicity. Sally series of questions grew out of an
  elementary observation of S. Abhyankar (\cite{Abhyankar}):
  \[ n \leq \rme + d-1.\]
    This is proved by first assuming, harmlessly, that the residue field of $\RR$ is infinite and modding out
  $\RR$ by a minimal reduction $Q$ of $\m$.  The inequality results from $\lambda(\m/Q) \geq \nu(\m/Q)$, the fact
  that $Q$ is generated by a subset of a minimal generating set of $\m$
   and the
  interpretation of 
 the multiplicity   $\rme$ as  $\lambda(\RR/Q) $.

\medskip

 The task undertaken by Sally in a series of papers (\cite{Sal77, 
 Sal79, Sal80, Sal80a, Sal78})
was to examine the impact of extremal values of this inequality on the properties of the associated
graded ring $G=\gr_{\m}(\RR)$. Thus in \cite{Sal77} it is proved that
if $   n = \rme + d-1$, then $G$ is Cohen--Macaulay. In the next case,
if 
  $ n = \rme + d-2$ or $n=\rme + d-3$,
   $G$ is still Cohen--Macaulay if $\RR$ is Gorenstein.

\medskip

If $   n = \rme + d-2$ numerous cases showed that $G$ is not always Cohen--Macaulay but 
nevertheless that still $\depth G \geq d-1$. That $\depth G \geq d-1$ for all Cohen--Macaulay rings
coalesced into a question that became known as the {\em Sally Conjecture}. It was independently 
resolved in the affirmative by M. E. Rossi and G. Valla (\cite{RV96}) and H.-J. Wang (\cite{Wang97}).

\medskip

 It should be pointed out that these developments have been greatly enhanced for more general
filtrations in \cite{RV10}. Another aspect muted here, for which this author apologizes, is the lack of 
a discussion of the Hilbert of function of $G$, which was of great interest to the authors mentioned above.   

\medskip




 
 

 

 
 




\section{Equimultiple Ideals}
\noindent
Let $\mathcal{F}$ be a multiplicative filtration and let $\AA =\RR[Q\TT]$ be a reduction of $\BB$.
If $ Q$ is not $\m$--primary, but still a complete intersemection, that is $I_1$ is an equimultiple ideal,
we would need an extended notion of Hilbert polynomial. Of
course this requires an understanding of the relationship between the coefficients and the cohomology of $S_{\BB/\AA}$.
%
Let us outline two ways  that  might  work out. The first uses an extended degree function (\cite{icbook}), the other a generalized multiplicity.
They are both focused on determining the multiplicity of $S_{\BB/\AA}$.

\subsection{Extended Hilbert function}

We are just going to replace lengths by multiplicities. Let $\mathcal{F} = \{I_n, \ I_0=\RR\}$ be a filtration of ideals.
Let $\deg(\cdot)$ be a multiplicity function, possibly a Hilbert--Samuel  multiplicity [usually denoted by $\rme_0(I;\cdot )$]
 Define
\[ n\mapsto \deg (\RR/I_n),\]
and set
\[ H(\RR;\ttt) = \sum_{n\geq 1} \deg(\RR/I_n) \ttt^n,  \quad H(\gr(\BB); \ttt) = \sum_{n\geq 0}\deg(I_n/I_{n+1}) \ttt^n.
\]

\begin{Proposition} Let $Q$ be an ideal of codimension $r$. Then
$H(\RR;\ttt)$ is a rational function of degree $r+1$.
\end{Proposition}

\begin{proof} Let $\pp_1, \ldots, \pp_s$ be the minimal prime ideals of $Q$. Since $Q$ is a reduction of $I_1$, they are also
the minimal primes of height $r$ of the ideals $I_n$. By the  elementary additivity formula for multiplicities,
\begin{eqnarray} \label{adddeg}  \deg(\RR/I_n) = \sum_{1\leq k\leq s} \deg (\RR/\pp_k) \lambda((\RR/I_n)_{\pp_k}).
\end{eqnarray}
The assertion follows because
$\lambda((\RR/I_n)_{\pp_k}) $ is the Hilbert function of the localization of the filtration $\{I_n\}$ in the local ring
$\RR_{\pp_k}$ of dimension $r$.
 \end{proof}

We note that when written out,  the series
\[ \sum_{n\geq 0} \lambda((I_n/I_{n+1})_{\pp_k}) \ttt^n = {\frac{h_k(\ttt)}{(1-\ttt)^{r}}}, \quad h_k(1) >0.\]

We collect the first Hilbert coefficients of these localizations as
\begin{eqnarray*}
E_0(\gr(\BB)) &=& \sum_{1\leq k\leq s} \deg(\RR/\pp_k) e_0(\gr(\BB_{\pp_k}))\\
E_1(\gr(\BB)) & = & \sum_{1\leq k \leq s} \deg(\RR/\pp_k) e_1(\gr(\BB_{\pp_k}))
\end{eqnarray*}
which we make use of to write
 the corresponding Hilbert polynomial as
\[ E_0(\gr(\BB)) {{n+r-1}\choose {r-1}}- E_1(\gr(\BB)){{n+r-2} \choose{r-2}} + \textrm{lower terms}.\]
 The coefficients are integers and are  called the {\em extended} Hilbert coefficients of $\gr(\BB)$.
 
\medskip

Now we use this technique to determine some properties of the corresponding Sally module $S_{\BB/\AA}$ in case $Q$ is a 
complete intersection of codimension $r$ and $I_n=I^n$ where $I$ is an unmixed ideal. 
Note that if $I$ is $\m$-primary,
 $E_0(\gr(\BB)) = \rme_0(I)$
and
 $E_1(\gr(\BB)) = \rme_1(I)$.

\begin{Proposition} \label{nonzeroSally} Let $\RR$ be a Cohen--Macaulay local ring of dimension $d\geq 1$.
If $Q$ is a complete intersection and $I$ is unmixed then
$S_{\BB/\AA} =0$ or $\dim S_{\BB/\AA}=d$. In the latter case $S_{\BB/\AA}$ has the $S_1$ property of Serre.
\end{Proposition}

\begin{proof}
 Start by noticing that $\RR[Qt] \otimes \RR/I= \RR[Qt]\otimes \RR/Q\otimes \RR/I$ is a polynomial ring over
$\RR/I$, and therefore has the condition $S_1$ since $I$ is unmixed. This means that $I\RR[Qt]$ is a divisorial
ideal of the Cohen-Macaulay ring $\AA=\RR[Qt]$. 
If $S_{\BB/\AA}\neq 0$ and $P\subset \AA$ is one of its associated primes of codimension at least $2$, the exact 
sequence
\[ 0\rar \Hom(\AA/P, I\AA)= 0 \lar \Hom(\AA/P, I\BB) = 0 \lar \Hom(\AA/P, S_{\BB/\AA}) \lar \Ext_{\AA}^1(\AA/P, I\AA)\]
leads to  a contradiction since the last module vanishes as $I\AA$ has the condition $S_2$ of Serre.
\end{proof}

\subsection{Ideals of positive dimension}

Let $I$ be of dimension one and examine how $S_{\BB/\AA}$ is Cohen-Macaulay. [Set $J= Q$] In the exact sequence
\[ 0 \rar I\RR[QT] \lar I \RR[IT] \lar S_Q(I) \rar 0,\]
$S_Q(I)\neq 0$ is Cohen-Macaulay means that $\depth I \RR[IT] \geq d$, which from the basic sequences mean that
$\depth \RR[IT] \geq d$ also [recall $I\RR[QT]$ is MCM $\RR[QT]$-module.  A maximal sop for $S_Q(I)$ is
$(QT, a)$, $a $ a sop for $\RR/Q$. Reduction mod $QT$ gives rise to the exact sequence 
\[ 0 \rar I\AA\otimes_{\AA} \AA/(QT) \lar I \RR[IT] \otimes_{\AA} \AA
/(QT) \lar S_Q(I) \otimes_{\AA}\AA /(QT) \rar 0.\]
The last module is
\[ F_Q(I)= \bigoplus_{n=1}^r I^{n+1}/QI^n, \quad r = \red_Q(I)\] must be Cohen-Macaulay over $\RR/Q$.
In particular

\begin{Proposition} $\red_Q(I) \leq \deg F_Q(I)$.
\end{Proposition}



\begin{Corollary} If $S_Q(I) $ is Cohen-Macaulay then $\deg S_Q(I) = \deg F_Q(I)$. 
\end{Corollary}

This gives a curious fact [probably trivial!]

\begin{Proposition} For all $n$, $I^n/QI^{n-1}$ is a torsion free $\RR/Q$-module.
\end{Proposition}

Let us examine the possible extension to the case of an ideal Cohen-Macaulay ideal $J$ of codimension two and deviation one, generically a complete intersection. This will imply that $\AA=\RR[Jt]$ is
Cohen-Macaulay. 

\medskip

 The extended Hilbert series
of $S_{\BB/\AA}$,
$ \sum_{n\geq 2} \deg(I^n/IJ^{n-1})\ttt^n$, can be related to that of $\gr(\BB)$ by comparison of the exact sequences
\[ 0 \rar I^n/IJ^{n-1} \lar \RR/IJ^{n-1} \lar \RR/I^n \rar 0,\]
and 
 \[ 0 \rar J^{n-1}/IJ^{n-1} = J^{n-1}/J^n \otimes \RR/I \lar \RR/IJ^{n-1} \lar \RR/J^{n-1} \rar 0\]
of  modules of dimension $d-r$. 

\medskip

These sequences give rise to the  equalities
\begin{eqnarray*}
\deg(I^n/IJ^{n-1}) & = & \deg (\RR/IJ^{n-1}) - \deg(\RR/I^n) \\
\deg(\RR/IJ^{n-1}) & = & \deg(J^{n-1}/IJ^{n-1}) + \deg(\RR/J^{n-1}).
\end{eqnarray*} 
They give for $n\geq 2$,
\begin{eqnarray*}
\deg(I^n/IJ^{n-1}) &=& \deg(J^{n-1}/IJ^{n-1}) + \deg(\RR/J^{n-1})-\deg(\RR/I^n)\\
&=& \deg(\RR/I) {{n+r-2}\choose{r-1}} + \deg(\RR/J){{n+r-2}\choose{r}} \\
&-&
 [E_0(\gr(\BB)) {{n+r-1}\choose {r}}- E_1(\gr(\BB)){{n+r-1} \choose{r-1}} + \textrm{lower terms}].\end{eqnarray*}
In particular
\begin{eqnarray*}
E_0(\gr(\BB)) & = & \deg(\RR/J) \\
E_1(\gr(\BB)) & \geq & 0.
\end{eqnarray*}

\bigskip

 We can also define an extended Hilbert function for the Sally module $S_{\BB/\AA}$
 and the corresponding Hilbert polynomial
\[\deg(I^n/IJ^{n-1})=
 E_0(S_{\BB/\AA}) {{n+r-1}\choose {r-1}}- E_1(S_{\BB/\AA}){{n+r-2} \choose{r-2}} + \textrm{lower terms}.\]
 The coefficient $E_0(S_{\BB/\AA})\geq 0$ but may vanish, that is this Hilbert polynomial could be of degree $<r$.
 
\begin{Proposition}\label{Sallydimnonzero} 
Let $J$ be an unmixed ideal generated by a $d$-sequence with a Cohen-Macaulay Rees algebra. Then

\begin{enumerate}

\item 
$ E_0(S_{\BB/\AA}) = E_1(\gr(\BB)) - \deg (I/J)$.

\item $E_0(S_{\BB/\AA}) = 0$ if and only if $\dim S_{\BB/\AA}< d$.

\item If $\dim S_{\BB/\AA} = d$, $E_0(S_{\BB/\AA}) = \deg S_{\BB/\AA}$.


\end{enumerate}

\end{Proposition} 

\begin{proof} To treat the vanishing of $E_0(S_{\BB/\AA})$, 
from the expression for the Hilbert polynomial of the filtration $I^n$,
\begin{eqnarray*} E_0(S_{\BB/\AA})=
E_1(\gr(\BB)) - \deg(I/J) &=& \sum_k[e_1(\gr(B_{\pp_k})) - \lambda(I_{\pp_k}/J_{\pp_k})]\deg(\RR/\pp_k)\\
&=& \sum_k \deg S_{\BB_{\pp_k}/\AA_{\pp_k}}\deg(\RR/\pp_k).
\end{eqnarray*}
Therefore $E_0(S_{\BB/\AA})=0$ if and only if the Sally modules at the localizations $\RR_{\pp_k}$ all vanish according to Proposition~\ref{SallyJci}. 
\end{proof}


 Let us  interpret these coefficients.

\begin{Proposition} $E_0(S_{\BB/\AA})= \deg S_{\BB/\AA}$.
\end{Proposition}

\begin{proof} This follows from (\ref{adddeg}) and the general  associativity formula for multiplicities (\cite[Theorem 24.7]{Nagata}). \end{proof}

\subsection{Cohen-Macaulayness}

This allows to state the Cohen--Macaulay test for $S_{\BB/\AA}$ as follows:

\begin{Theorem}[CM test for Sally modules] Suppose $J$ is a complete intersection.
Let $\aa=\{a_1, \ldots, a_{d-r}\}$ be a multiplicity set of parameters
for $\RR/J$, that is $\deg \RR/J = \lambda(\RR/(J, \aa))$. If $I_1$ is a Cohen--Macaulay ideal and  $S_{\BB/\AA}\neq 0$ then
\[ E_1(\gr(\BB)) \leq  \sum_{j\geq 1} \lambda(I_j/((\aa) I_j + JI_{j-1})),\]
with equality if and only if $S_{\BB/\AA}$ is Cohen--Macaulay.
\end{Theorem}

It is important to keep in mind that $Jt, \aa$ is a system of parameters for both $\gr(\BB)$ and $S_{\BB/\AA}$.
To explain this formulation, just note that the right-hand side is the sum of the multiplicity of
the Cohen--Macaulay $\RR/J$--module $I_1/J$ (recall that $I_1$ is a Cohen--Macaulay ideal) plus
$\lambda(S_{\BB/\AA} \otimes \AA/(\aa, Jt))$. One then invokes  Serre's Theorem.

\bigskip

To make it effective requires information about $E_1(\gr(\BB))$, which is not an easy task. We will make
 observations about the following topics:

\begin{Proposition} With $\BB$ as above, let   $\CC$ be its integral closure. Then
\begin{enumerate}
\item[{\rm (1)}] $\BB$ satisfies the condition $R_1$ of Serre if and only if
\begin{eqnarray*}
E_1(\gr(\BB)) &=& E_1(\gr(\CC)).
\end{eqnarray*}
\item[{\rm (2)}]  If $J$ is $\m$--primary, these conditions are equivalent to
\begin{eqnarray*}
\deg {F}(\BB) = \deg {F} (\CC).
\end{eqnarray*}

\end{enumerate}
\end{Proposition}

\subsection{Ideals of positive dimension: non Cohen-Macaulay rings}
Let $(\RR,\m)$ be a 
Noetherian local ring of dimension $d>0$ that satisfies the condition $S_r$ of Serre, $r\geq 1$.
 For the study of Sally
modules of ideals $I$ with $\dim \RR/I > 0$ we need an understanding of the Rees algebras of some
of their reductions. Let us consider some special cases of equimultiple ideals.  We make use of the techniques of
 \cite[Theorems 6.1, 9.1 and 10.1]{HSV3II}).

\medskip
\begin{enumerate}

\item $Q = (x_1, \ldots, x_q)$ is a partial system of parameters and a $d$-sequence. 
 Since $Q$ is generated by a $d$-sequence, the approximation complex  
 (set \ $\SS=\RR[\TT_{1}, \ldots, \TT_{q}]$)
\[ 0\rar \H_q(Q)\otimes \SS[-d] \rar \cdots \rar \H_1(Q)\otimes \SS[-1] \rar  \H_0(Q)\otimes \SS\rar \gr_Q(\RR) \rar 0  \]
is acyclic (\cite[Theorem 5.6]{HSV3}). As
 $\RR$ has the condition $S_r$,  $\H_j(Q) = 0$ for $j >  q-r$, which gives the estimate
 \[ \depth \gr_Q(\RR) \geq q - (q-r) = r.\]   

\item By tensoring this complex by $\RR/I$, we get the exact 
sequence
\[  \H_1(Q) \otimes \RR/I \otimes \SS[-1] \stackrel{\phi}{\lar} \RR/I \otimes \SS
 \lar \overline{G}= \gr_Q(\RR) \otimes \RR/I  \rar 0.\]
 If  $I_1(\varphi)\subset I$,  the mapping $\phi$ is trivial so that $\overline{G}=\RR/I
 \otimes \SS$,
 which gives 
 \[ \depth \overline{G} \geq q.\]

\item Now we look at the approximation complex
\[ 0\rar Z_q(Q)\otimes \SS[-q] \rar \cdots \rar Z_1(Q)\otimes \SS[-1] \rar  Z_0(Q)\otimes \SS\rar \RR[Q\TT] \rar 0 , \] which is also acyclic. Here the $Z_i$ are the modules of cycles of the Koszul complex of $Q$. We have
\begin{eqnarray*}
Z_q & = & 0\\
\depth Z_{q-1} & \geq & r \\
   & \vdots & \\
 \depth Z_i & \geq &  r-i+1, \quad  i > q-r
 \\
 \depth Z_0 & \geq & r.
 \end{eqnarray*}  
 Depth chasing the exact sequence yields
\begin{eqnarray*}  
\depth \RR[Q\TT] &\geq & r+1,
\end{eqnarray*}
in all cases. This  means that $\RR[Q\TT]$ satisfies $S_{r+1}$. 
 
 \medskip
\item Suppose $Q$ is a minimal reduction of $I$ and consider the corresponding Sally module
\[ 0 \rar I \RR[Q\TT] \lar I \RR[I\TT] \lar S_Q(I) \rar 0.\]
To get information about $\depth I \RR[Q\TT]$, the exact sequence
\[ 0 \rar I\RR[Q\TT] \lar \RR[Q\TT] \lar \overline{G} \rar 0,\]
and  (1) and (2) gives $S_Q(I)$ has the condition $S_1$ or is zero.
\medskip

\item If $\RR[I\TT]$ has the condition $S_2$ then so does $S_Q(I)$.

\medskip

\item We need a version of Proposition~\ref{Sallydimnonzero} for ideals such as $Q$ [which may be mixed].

\medskip
\item
To approach a calculation of $s_0(S_Q(I))$, consider 
 the exact sequences
\[ 0 \rar I^{n+1}/IQ^{n} \lar \RR/IQ^{n} \lar \RR/I^{n+1} \rar 0,\]
and 
 \[ 0 \rar Q^{n}/IQ^{n} = \SS_n \otimes \RR/I \lar \RR/IQ^{n} \lar \RR/Q^{n} \rar 0\]
of  modules of dimension $\leq 1$. 
Applying the functor $\H_{\m}(\cdot)$ to these sequences, we get
\[ 0 \rar I^{n+1}/IQ^{n} \lar \H_{\m}^0( \RR/IQ^{n}) \lar \H_{\m}^0(
\RR/I^{n+1}) \rar 0,\]
and 
 \[ 0 \rar  
 \SS_n \otimes \H_{\m}^0(\RR/I) 
   \lar 
   \H_{\m}^0(\RR/IQ^{n}) \lar 
   \H_{\m}^0(\RR/Q^{n}) \lar  
   \SS_n \otimes \H_{\m}^1(\RR/I) .
    \]

\end{enumerate}

\subsection{$j$-multiplicities}
Let $(\RR, \m)$ be a Noetherian local ring, $\AA$ a standard graded $\RR$-algebra 
 and $M$ a finitely generated graded $\AA$-module. 
We can attach to $\H=\H_{\m}^{0}(M)$ a Hilbert function
\[ n \mapsto \lambda(\H_n),\]
which we call the $j$--{\em transform} of $M$.

\subsubsection*{Hilbert coefficients of Achilles-Manaresi polynomials}
The corresponding Hilbert series and Hilbert polynomial will be still written as
$P(M; \ttt)$ and $H(M; \ttt)$.
We use a different notation for the coefficients of these functions:

\[ H(M;\ttt) = \sum_{i = 0}^{r-1} (-1)^i j_i(M) {{\ttt+ r-i-1}\choose{r-i-1}}, \quad r=\ell(M). \]

 If $r=1$ this polynomial does not provide for $j_1(M)$, so we use instead
 the function
\[ n \mapsto \sum_{k\leq n} \lambda(\H_k).\]

The coefficients $j_i(M)$ are integers but unlike the usual case of an Artinian local $\RR$ it is
very hard to calculate being less directly related to $M$.
In addition, some general relationships that are known to exist between the standard coefficients
$\e_0, \e_1, \e_2$, for instance, are not known.

\medskip

We illustrate one of these issues with a series of  questions. Let $\RR$ be a Noetherian local ring and 
let $I=(x_1, \ldots, x_r)$, $r\leq d = \dim \RR$, be an ideal generated by a partial system of 
parameters. Let $G$ be the associated graded ring of $I$, $G= \gr_I(\RR)$. The module
$\H= \H_{\m}^0(G)$ has dimension $\leq r$.

\section{Ideals of Dimension One}

Our data is now the following kind:
\begin{itemize}
\item $\RR$ is Cohen--Macaulay of dimension $d\geq 1$
\item $I_1$ is a Cohen-Macaulay equimultiple ideal of dimension $1$, or
\item $J$ is an aci and $I$ is an ideal such that $I\Rees(J)$ is CM--HOW TO FIND SUCH?
\end{itemize}

\begin{itemize}
\item Setting up the Hilbert function in case of low reduction number

\item Role of the Sally fiber

\medskip

Here is a specific problem: Suppose $\RR = k[x,y,z]$ and $J = (a,b)$ is a homogeneous complete intersection. We want to find $I\subset \bar{J}$ such that $I\RR[J\TT]$ is Cohen-Macaulay. For
example, what if $I$ is the integral closure itself? 

\medskip

Of course, if $\RR[I\TT]$ is Cohen-Macaulay [which is the case if $J$ is monomial], then all
powers $I^n$ are Cohen--Macaulay, by standard factors: $\gr_I(\RR)$ is Cohen--Macaulay of dimension $3$ which the ideal $\m \gr_I(\RR)$ has height $1$ since the analytic spread
is $2$ and thus $\m$ is not an associated prime of any component $I^n/I^{n+1}$, then by induction get that $\m$ is not associated to any $\RR/I^n$.

\medskip

If $I$ is Cohen-Macaulay, get $I\RR[J\TT]$ CM and the Sally module behaves as in the dimension zero case [$\m$ primary ideal].

\medskip

The more interesting case is that of an aci $J$ with $I$ CM. In this case $\RR[J\TT]$
 is still CM, but not sure $I\RR[J\TT]$ is CM. Need examples.

\end{itemize}

\subsection{Generic complete intersections}
 Let $\RR$ be a Cohen-Macaulay local ring of dimension $d>1$ and $I$ a Cohen-Macaulay ideal
with $\dim \RR/I \geq 1$. We assume that $I$ is not equimultiple. 
One  special
target is the set of prime ideals of dimension one  of a regular local ring. 

\medskip

Let $I$ be an ideal that is a complete intersection on the punctured spectrum.
If $Q$ is a minimal reduction the Sally module $S_Q(I)$
has finite support and therefore has an ordinary Hilbert function. We do not know well how 
to express its Hilbert coefficients and use them to determine the properties of $S_Q(I)$.

\medskip

Let us assume $\dim \RR/I=1$, and
let $Q$ be a minimal reduction. We examine the Sally module $S_Q(I)$ under he assumption that $I$ is generically a complete intersection. In this case $Q$ is generated by a $d$-sequence and the approximation complex
\[ 0 \rar \H_1(Q) \otimes \SS[-1] \lar \H_0(Q) \otimes \SS \lar \gr_Q(\RR) \rar 0\]
is exact (see \cite[Theorem 9.1, 10.1]{HSV3II}). Furthermore, by \cite[Theorem 6.1]{HSV3II}, both
$\gr_Q(\RR)$ and $\RR[Q\TT]$ are Cohen--Macaulay.
Note that the Koszul homology module $\H_1(Q)$ is a Cohen-Macaulay module of dimension $1$ and $I\cdot  \H_1(Q )=0$ since
$I_{\p} = Q_{\p}$ at any of its minimal primes and $I_{\p}$ is generated by a regular sequence.
It follows that if we reduce this complex by $\RR/I$ we get the complex
\begin{eqnarray} \label{grI}
 0 \rar \H_1(Q) \otimes \SS[-1] \lar \RR/I \otimes \SS \lar \gr_Q(\RR)\otimes \RR/I \rar 0,
 \end{eqnarray}
that localization at the primes $\p$ shows it is acyclic. 

\begin{Proposition}\label{Sallygenci}
Under these conditions we have:
\begin{enumerate} 
\item
$\gr_Q(\RR)$ and 
$ \gr_Q(\RR)\otimes \RR/I$ are Cohen-Macaulay.
\item
If the Sally module is nonzero then $\dim S_Q(I) = d$ and has the condition $S_1$ of Serre.
\end{enumerate}
\end{Proposition}

\begin{proof}
The first assertion follows from the sequence (\ref{grI}), as the other modules are Cohen-Macaulay
of dimension $d+1$.

\medskip 

Now consider the Sally module
\[ 0 \rar I \RR[Q\TT] \lar I \RR[I\TT] \lar S_Q(I) \rar 0.\]
Since $I \RR[Q \TT]$ is a maximal Cohen-Macaulay module, $S_Q(I) = 0$ or has the condition $S_1$ of Serre.
\end{proof}

\begin{Remark}{\rm With greater generality, suppose $\RR$ is a Gorenstein local ring and $I$
is a Cohen-Macaulay ideal of codimension $r$ and analytic deviation $1$, that is a minimal reduction $Q$
is generated  by $r+1$ elements. If $I$ is a complete intersection on the punctured spectrum the same
assertions about $S_Q(I)$ will hold.
}\end{Remark}

These Sally modules have properties akin to those of $\m$-primary ideals.

\begin{Corollary} Let $\RR$ and $I$ be as above. If $\dim S_{\AA}(\BB) = d$,  then
\begin{enumerate}


\item
 $\p =\m \AA$ is the only associated prime of $S_{\AA}(\BB)$.

\item The multiplicity of $S_{\AA}(\BB)$ satisfies 
\[ \rme_0(S_Q(I)) = \lambda( S_{Q}(I)_{\p}) 
 \leq \sum_{n=1}^{r-1} \lambda(I^{n+1}/QI^n),\]
with equality if it is  Cohen--Macaulay.
\end{enumerate}

\end{Corollary}

\begin{proof} 
Since $S_Q(I)_{\q}=0$ for $\q \neq \m$, $S_Q(I)$ is annihilated
 by a power of $\m$. As it has the condition $S_1$, $\p$ is indeed its unique associated prime.
 As a consequence $(Q\TT)$ gives a system of  parameters for $S_Q(I)$. In particular
\[ \rme_0(S_Q(I)) 
 \leq \sum_{n=1}^{r-1} \lambda(I^{n+1}/QI^n),\]
 by Serre's formula on Euler characteristics. 
\end{proof}

Unfortunately neither the multiplicity nor the reduction number have estimates.

\medskip
{\bf Reduction number:} 
{\bf Stuff below  is not entirely well. Will return.}

\medskip

To approach a calculation of $s_0(S_Q(I))$, consider 
 the exact sequences
 \begin{eqnarray}
 0 \rar I^{n+1}/IQ^{n} \lar \RR/IQ^{n} \lar \RR/I^{n+1} \rar 0,\\
   0 \rar Q^{n}/IQ^{n} = \SS_n \otimes \RR/I \lar \RR/IQ^{n} \lar \RR/Q^{n} \rar 0
   \end{eqnarray}
of  modules of dimension $\leq 1$. 
Applying the functor $\H_{\m}(\cdot)$ to these sequences, we get
\begin{eqnarray}
 0 \rar 
I^{n+1}/IQ^{n} = \H_{\m}^0(I^{n+1}/IQ^{n} )
\lar \H_{\m}^0( \RR/IQ^{n}) \lar \H_{\m}^0(
\RR/I^{n+1}) \rar 0,\\
 0 \rar  
 \SS_n \otimes \H_{\m}^0(\RR/I) =0
   \lar 
   \H_{\m}^0(\RR/IQ^{n}) \lar 
   \H_{\m}^0(\RR/Q^{n}) \lar  
   \SS_n \otimes \H_{\m}^1(\RR/I) .
    \end{eqnarray}
\medskip

We need to know about the   growth of 
\begin{eqnarray*}
s(n) &= &\lambda(I^{n+1}/IQ^n)\\
a(n) &= & \lambda(\H_{\m}^0(\RR/Q^n))\\
b(n) &= &\lambda(\H_{\m}^0(\RR/Q^{n+1}))\\ 
c(n) & = & \lambda(\H_{\m}^0(\RR/I^{n+1}))\\
d(n) &= & \lambda(
\H_{\m}^0(\RR/
IQ^n))\\
e(n)  & = &\lambda(\H_{\m}^0(I^{n+1}/Q^{n+1}))\\  f(n) &= & \lambda(\H_{\m}^0(Q^{n}/Q^{n+1}))\\
g(n) &= & \lambda(\H_{\m}^0(Q^{n}/IQ^{n}))
\end{eqnarray*}
by making use of the approximation complexes. These are for $n\gg 0$
Hilbert   polynomials of 
degree $<d$ and their coefficients denoted in the usual manner. Let us see what numerical data can be extracted 
from them.

\medskip

\begin{enumerate}

\item
By local duality we have $\Hom(\H_{\m}^1(\RR/I), E) = \RR/I$, where $E$ is the injective envelope of $\RR/\m$.
\medskip

\item The first two sequences yield
\[ s_0 = d_0 - c_0 < d_0 \quad \mbox{\rm except if }\]
and 
\[ d_0 \leq a_0,
\]
and therefore
\[ s_0 \leq a_0.\]

\item  From the exact sequence
\[ 0 \rar \H_1(Q) \otimes \SS[-1] \lar \H_0(Q) \otimes \SS \lar \gr_Q(\RR) \rar 0,\]
we have
\[  \H_{\m}^0(\H_1(Q)) 
\otimes \SS[-1] =0 \rar \H_{\m}^0(\H_0(Q)) \otimes \SS = I/Q \otimes \SS
\rar \H_{\m}^0(
\gr_Q(\RR)) \rar  \H_{\m}^1(\H_1(Q))\otimes \SS[-1] .\]
From which we have
\[ f_0 \geq \lambda(I/Q).\]
 
 \medskip
\item Meanwhile from
\[ 0 \rar Q^n/Q^{n+1} \lar \RR/Q^{n+1} \lar \RR/Q^n \rar 0\]
we have
\[ a_0 +  f_0 \geq b_0.\]

\item Let us exploit the coarse inequality $s_0 \leq a_0$ in some special cases. Suppose $\RR$ is a two-dimensional 
Gorenstein local ring. In this case we have the $Z$-complex
\[ 0 \rar Z_1 \otimes \SS_{n-1} \stackrel{\varphi}{\lar} Z_0 \otimes \SS_n \lar Q^n \rar 0,\]
which dualizing gives
\[ 0 \rar (Q^n)^{*} \lar Z_0^{*} \otimes \SS_{n}^* \stackrel{\varphi^*}{\lar} Z_1^{*} \otimes \SS_{n-1}^* \lar \Ext^1(Q^n, \RR) = \Ext^2(\RR/Q^n, \RR) \rar 0.\]

Note how $\varphi$ acts: If $Q=(a,b)$ and $(r,s) \in Z_1$, that is $ra + sb=0$,  then $\varphi(r,s) = r\TT_1 + s\TT_2$, and thus \[\varphi^*(\TT_1^*)((r,s)h(\TT_1,\TT_2))=\TT_1^*(r\TT_1+s\TT_2) h(\TT_1,\TT_2)= r
h(\TT_1,\TT_2), \]
and similar actions. This shows that
\[ \coker \ \varphi = \RR/I_1(\varphi) \otimes \SS_{n-1}.\]

\item To sum up  \[ s_0 \leq a_0 = \lambda(\RR/I_1(\varphi)).\]

\medskip

\item {\bf This is incomplete:} Must use the sequence $0\rar B_1 \lar Z_1 \lar \H_1 \rar 0$, dualize to get
$ 0 \rar Z_1^* \lar B_1^*\lar \RR/I \rar 0$

\end{enumerate}



\medskip
\subsubsection*{Special cases}
 Suppose $\nu(I) = d+1$ and consider the sequence (see (\ref{cpv6}))
\begin{eqnarray*}\label{sequence3}
F(Q) \oplus F(Q)[-1] \stackrel{\overline{\varphi}}{\lar} {F}(I) \lar  S_Q(I)[-1]\otimes \RR/\m
 \rar 0.\end{eqnarray*}
 
If $\m S_Q(I)=0$, we examine the image of $\overline{\varphi}$ as in \cite[Proposition 3.1]{red}.
We want to argue that $\overline{\varphi}$ is injective. It suffices to show that its image $C$ in
$F(I)$ has rank $2$.
We begin by observing that  $F(Q)$ injects into  $F(I)$, so consider the exact sequence
\[ 0 \rar F(Q)  \lar C \lar D \rar 0,\]
where
\[ D= \bigoplus_{n\geq 1} (IQ^{n-1} + \m I^n)/(Q^n+ \m I^n).\] 

We
examine its Hilbert function,
$ \lambda((IQ^{n-1}  + \m I^n)/(Q^n + \m I^n) )$ for $n>>0$ to show that $\dim D = d$. First
note that by assumption $\m S_Q(I)=0$, that is $\m I^n \subset IQ^{n-1}$ for all $n\geq 2$. Of course if $I = (Q,a)$
it suffices to assume $n=2$.
We want to argue that $\m I^n = \m IQ^{n-1}$.
This is embedded in the proof of the following (\cite[Proposition 3.1]{red}):


\begin{Proposition}\label{redlaq2a}  Let $(\RR, \m )$  be a Noetherian local ring 
  with infinite residue field,  $I= (Q,a)$ is an ideal  and  $Q$ is one of its minimal reductions. If for some 
 integer $n$, $\lambda(I^n/QI^{n-1})=1$, then $\red(I) \leq n \nu(\m) -1$.
\end{Proposition}

\begin{proof} Since $Q$ is  a minimal reduction of $I$, it is generated by a subset of the  minimal set of generators of $I$ so if $\nu(Q) = r$
the expected number of generators of $I^n$ is ${r+n}\choose{n}$. A lesser value for $\nu(I^n)$ would imply by
\cite[Theorem 1]{ES} that $I^n = JI^{n-1}$ for some minimal reduction $J$ of $I$, and therefore $\red(I) \leq n-1$.

\medskip
 Suppose that $\nu(I^{n}) ={{r+n}\choose{n}}$. Since $\lambda(I^n/QI^{n-1})=1$, $\m I^{n} \subset QI^{n-1}$. Moreover, we have
\[ \m QI^{n-1} \subset \m I^{n} \subset QI^{n-1} \subset I^{n}.\]
Note that
\[ \lambda( \m I^{n}/ \m Q I^{n-1} ) = \lambda(I^{n}/\m QI^{n-1}) - \lambda(I^{n}/\m I^{n}) = \lambda(I^{n}/QI^{n-1}) + \nu(QI^{n-1}) - \nu(I^{n})=0.\]
It follows that $\m I^n = \m QI^{n-1}$. From the Cayley-Hamilton theorem we have
\[ (I^n)^s = (QI^{n-1})(I^{n})^{s-1},\] where $s = \nu(\m)$.  In particular $I^{ns} = QI^{ns-1}$, as desired. 
\end{proof}

We return to the  calculation of the dimension of $D$.  
Since $\m I^n = \m Q^{n-1}I$,
  we have
  \begin{eqnarray*}
 \lambda((IQ^{n-1}  + \m I^n)/(Q^n + \m I^n) )& =& \lambda(Q^{n-1}/(Q^n + \m I^n)) - \lambda(Q^{n-1}/(IQ^{n-1} + \m I^n))\\
 &=& \lambda((Q^{n-1}/Q^n)\otimes \RR/(Q+\m I))
 - \lambda((Q^{n-1}/Q^n)\otimes \RR/ I) \\
 &=& \lambda(I/(Q+\m I)) {{n+d-2}\choose{d-1}},
 \end{eqnarray*}
  which shows that $D$ has dimension $d$ and multiplicity $\lambda(I/(Q+\m I))=1$. This concludes the proof that
  \[ f_0(I) = 2 + s_0(S_Q(I)).\]

  We are now ready for
  the main application of these calculations.
  
  \begin{Theorem} Let $(\RR, \m)$ be a Cohen-Macaulay local ring of  dimension $d$ and $I$ a Cohen-Macaulay ideal 
  of dimension  $1$. Suppose $\nu(I)=d+1$ and $I$ is generically a complete intersection. Let $Q$ be a minimal reduction  of $I$ and $S_Q(I)$ the associated Sally module. If $\m S_Q(I)=0$ then
  $F(I)$ is Cohen-Macaulay and $\depth S_Q(I)\geq d-1$. 
  \end{Theorem}
  
  \begin{proof} If the mapping $\varphi$ is injective, in the exact sequence
  \[ 0 \rar F(Q) \oplus F(Q)[-1] \lar F(I) \lar S_Q(I) \rar 0,\]
  $F(I)$ is torsionfree over $F(Q)$ since $S_Q(I)$ is also torsionfree over $F(Q)$. 
   The Cohen-Macaulayness of $F(I)$ 
  then follows as in Theorem~\ref{CHTheorem}(2).
\end{proof}

\begin{Remark}{\rm If $\nu(I) \geq d+2$, 

\medskip
\begin{itemize}
\item From $\m QI = \m I^2$, we still have $\red_Q(I) \leq 2 \cdot \nu(\m) -1$

\medskip

\item For $\nu(I) = d+2$ say, we have the complex
  \[  F(Q) \oplus F(Q)[-1]^2 \lar F(I) \lar S_Q(I) \rar 0,\]
but how to determine the kernel?
\end{itemize}

}\end{Remark}
 
\subsection{Dimensions two and three}  
  
We will begin by illustrating with elementary examples.

\begin{Example}{\rm 
Let $\RR = k[x,y,z]$ and $I  = (x^2 y, y^2 z, z^2 x, xzy)$. $I$  is Cohen-Macaulay and generically a complete intersection.
 $Q = (x^2 y - xyz, y^2 z, z^2 x-xyz)$ is a minimal reduction of $I$ and $\lambda(I^2/QI) =  1$.
Then $\red_Q(I) \leq 2$ (actually $\red_Q(I) = 2$).

}\end{Example}

\begin{Example}{\rm
Let $\RR$ be a regular local ring of dimension $3$ and $I$ a prime ideal of codimension two. Let $Q$ be a minimal 
reduction of $I$.
\medskip

{\bf What can be said?}

\medskip

\begin{itemize}

\item We may assume $\nu(Q) = 3$, as otherwise $I$ is a complete intersection.

\end{itemize}

}\end{Example}


\section{This and That}
We will report briefly on the structure of Sally modules of small rank.

\subsection{Multiplicity one} 
\cite{CPP98}, \cite{C09}, \cite{GNOmmj}, \cite{GNOmrl}, \cite{GO10}, \cite{VazPinto96}

\begin{Theorem}{\rm \cite[Theorem 1.2]{GNOmrl}}
Let $(\RR, \m)$ be a Cohen-Macaulay local ring of dimension $d$. If $I$ is a $\m$--primary ideal with a minimal
reduction $Q$
then the following three conditions are equivalent to each other.
\begin{enumerate}
\item[(1)] $\rme_{1} = \rme_{0} - \lambda(\RR/I) + 1$.
\item[(2)] $\m S = (0)$ and $\rank_{\mathcal{F}(Q)}(S)= 1$.
\item[(3)] $S \simeq (X_{1}, X_{2}, \ldots, X_{c}) \mathcal{F}(Q)$ as graded $\Rees(Q)$--modules for some $0 < c \leq d$, where $\{ X_{i} \}_{1 \leq i \leq c}$ are linearly independent linear forms of the polynomial ring $\mathcal{F}(Q)$.
\end{enumerate}
When this is the case, $c = \lambda(I^{2}/QI)$ and $I^{3} = QI^{2}$, and the following assertions hold true.
\begin{enumerate}
\item[{\rm (i)}] $\depth G \geq d-c$ and $\depth_{\tiny \Rees(Q)}S = d-c+1$.
\item[{\rm (ii)}] $\depth G = d-c$ if $c \geq 2$.
\item[{\rm (iii)}] Suppose $c <d $. Then
\[ \lambda(\RR/I^{n+1}) = \rme_{0}{{n+d}\choose{d}} - \rme_{1}{{n+d-1}\choose{d-1}} +  {{n+d-c-1}\choose{d-c-1}} \]
for all $n \geq 0$. Hence
\[ \rme_{i} = \left\{ \begin{array}{ll} 0 \quad &\mbox{\rm if} \;\; i \neq c+1 \\ (-1)^{c+1} \quad & \mbox{\rm if} \;\; i= c+1 \\ \end{array}  \right. \]
for $2 \leq i \leq d$.
\item[{\rm (iv)}] Suppose $c=d$. Then
\[ \lambda(\RR/I^{n+1}) =  \rme_{0}{{n+d}\choose{d}} - \rme_{1}{{n+d-1}\choose{d-1}} \]
for all $n \geq 1$. Hence $\rme_{i}=0$ for $2 \leq i \leq d$.
\end{enumerate}
\end{Theorem}

This result is very similar to the assumptions and the consequences of the {\em Sally conjecture}.

\begin{Corollary}
In addition, if $I$ is an almost complete intersection then the Sally module is Cohen-Macaulay.
\end{Corollary}

\subsection{Normal filtration}
\cite{CPR14}, \cite{OzRo15},
\cite{Phuong15}

\begin{Theorem}{\rm \cite[Lemma 2.3]{Phuong15}}
 Let $(\RR,\m)$ be an analytically unramified local ring satisfying Serre  condition
$S_2$ and let $\mathcal{F} = \{I_n\}$ be a filtration of ideals such that $\codim \ I_1\geq 2$ and 
$\BB = \bigoplus_{n\geq 0} I_n \TT^n$ is Noetherian. Then the integral closure of $\BB$ in $\RR[\TT]$ has the condition $S_2$.
\end{Theorem}

\begin{Theorem}{\rm \cite[Theorem 2.5]{CPR14}} Let $(\RR, \m)$ be an analytically umramified 
Cohen-Macaulay local ring of dimension $d$ and infinite residue field and let $I$ be an $\m$ primary
 ideal. If $
\mathcal{F}=  \{ I_n = \overline{I^n}\}$
 is the corresponding normal filtration, denote by $\BB$ its Rees algebra and by $\overline{S}$ its
 Sally module.
 Then 
 \begin{enumerate}
 \item If $\overline{S}\neq 0$ then $\dim \overline{S} = d$.
 
\item If $\overline{S}\neq 0$ then $\depth \gr(\BB) \geq  \depth \overline{S} -1$.

 \item If ${s_0}(\overline{S})=1$ then $\overline{S}$ is Cohen-Macaulay.
 
 \end{enumerate}
\end{Theorem}

\begin{Theorem}{\rm \cite[Theorem 1.2]{OzRo15}}
Let $(A, \m)$ be a Cohen-Macaulay local ring and $I$ an $\m$--primary ideal. Let $Q$ be a minimal reduction of $I$ and let 
\[ R= A[It],\;\; T=A[Qt],\;\; G=\gr_{I}(R),\;\; C=(I^{2}R/I^{2}T)(-1),\;\; B=T/\m T.\]
Assume that $I$ is integrally closed. Then the following conditions are equivalent:

\begin{enumerate}
\item[(1)] $\rme_{1}(I)= \rme_{0}(I) - \lambda_{A}(A/I) + \lambda_{A}(I^{2}/QI)+1$,

\item[(2)] $\m C =(0)$ and $\rank_{B}(C)=1$.

\item[(3)] $C \simeq (X_{1}, X_{2}, \ldots, X_{c})B(-1)$ as graded $T$--modules for some $1 \leq c \leq d$, where $X_{1}, \ldots, X_{c}$ are linearly independent linear forms of the polynomial ring $B$.
\end{enumerate}

\noindent When this is the case, $c= \lambda_{A}(I^{3}/QI^{2})$ and $I^{4}=QI^{3}$, and the following assertions hold true:
\begin{enumerate}
\item[{\rm (i)}] $\depth G \geq d-c$ and $\depth_{T} C = d-c+1$.

\item[{\rm (ii)}] $\depth G = d-c$ if $c \geq 2$.

\item[{\rm (iii)}] Suppose $c=1 <d$. Then $\mbox{\rm HP}_{I}(n) = \lambda_{A}(A/I^{n+1})$ for all $n \geq 0$ and
\[ \rme_{i}(I) = \left\{ \begin{array}{ll}   
\rme_{1}(I) - \rme_{0}(I) + \lambda_{A}(A/I) + 1 \;\; & \mbox{if} \; i=2 \\
1 & \mbox{if} \; i=3 \;\; \mbox{and} \;\; d \geq 3 \\
0 & \mbox{if} \; 4 \leq i \leq d. \\
\end{array}   \right.\]

\item[{\rm (iv)}] Suppose $2 \leq c <d$. Then $\mbox{\rm HP}_{I}(n) = \lambda_{A}(A/I^{n+1})$ for all $n \geq 0$ and
\[ \rme_{i}(I) = \left\{ \begin{array}{ll}   
\rme_{1}(I) - \rme_{0}(I) + \lambda_{A}(A/I)  \;\; & \mbox{if} \; i=2 \\
0 & \mbox{if} \; i \neq c+1, c+2, \;\; 3 \leq i \leq d \\
(-1)^{c+1} & \mbox{if} \;  i = c+1, c+2, \;\; 3 \leq i \leq d. \\
\end{array}   \right.\]

\item[{\rm (v)}] Suppose $c =d$. Then $\mbox{\rm HP}_{I}(n) = \lambda_{A}(A/I^{n+1})$ for all $n \geq 2$ and
\[ \rme_{i}(I) = \left\{ \begin{array}{ll}   
\rme_{1}(I) - \rme_{0}(I) + \lambda_{A}(A/I)  \;\; & \mbox{if} \; i=2, \;\; \mbox{and} \;\; d \geq 2 \\
0 & \mbox{if} \;  3 \leq i \leq d \\
\end{array}   \right.\]

\item[{\rm (vi)}] The Hilbert series $\mbox{\rm HS}_{I}(z)$ is given by
\[ \mbox{\rm HS}_{I}(z) =
\frac{ \lambda_{A}(A/I) + ( \rme_{0}(I) - \lambda_{A}(A/I) - \lambda_{A}(I^{2}/QI) -1  )z + ( \lambda_{A}(I^{2}/QI) + 1 ) z^{2} + (1-z)^{c+1}z  }{(1-z)^{d}}. \]

\end{enumerate}
\end{Theorem}

\subsection{Monomial ideals}

Sally modules of monomial ideals are hard to construct in a manner that inherit the monomial structure. More 
precisely if $I$ is a monomial ideal of $\RR = k[x_1, \ldots, x_d]$, its minimal reductions $ Q$ used
in the definition of the Sally modules of $I$ are rarely monomial. One should keep in mind that 
the point of using $Q$ was to guarantee a good platform in $\RR[Q\TT]$ from which to examine 
$\RR[I\TT]$. Of course this requirement may be satisfied in other cases, that is without having for $Q$ 
a minimal reduction.

\medskip

Let us illustrate with a special class of ideals.  
Monomial ideals of finite colength which are almost complete intersections have a very simple description. 
Let $\RR=k[x_1,\ldots ,x_d]$ be  a polynomial ring over  a (possibly infinite) field and let $J$ and $I$ be $\RR$--ideals such that
\[{\ds J=(x_1^{a_1},\; \ldots ,\; x_d^{a_d}) \subset (J,\; x_1^{b_1} \cdots x_d^{b_d})=I. }\] 
This is the general form of almost complete intersections of $\RR$ generated by monomials. Perhaps the most interesting cases are those 
where
${\ds \sum \frac{b_i}{a_i}  <1}$. This inequality
ensures that $J$ is not a reduction of $I$.
Let 
\[{\ds Q =(x_1^{a_1}-x_d^{a_d}, \ldots,
\; x_{d-1}^{a_{d-1}}-x_d^{a_d},\;x_1^{b_{d}}
\cdots  x_d^{b_d} ) }.\]

Note that $Q$ is a reduction of $I$: It is enough to observe that
\[ (x_d^{a_d})^d = x_{d}^{a_d}(
x_d^{a_d} - x_1^{a_1} +
x_1^{a_1}) \cdots (x_d^{a_d} 
 - x_{d-1}^{a_{d-1}} + 
x_{d-1}^{a_{d-1}})\in QI^{d-1}, \] 
 in particular $\red_Q(I) \leq d-1$.

 \medskip

 Let now $\BB $ be the Rees algebra of the integral closure filtration $\{ \bar{I^n}\}$. If 
 $\AA_0 = \RR[Q\TT]$ and $\AA= \RR[I\TT]$, either of these defines a Sally
 module, $S_{\AA_0}(\BB)$ and $S_{\AA}(\BB)$, the latter carrying a monomial structure. While 
 $\AA_0$ is  always Cohen--Macaulay, $\AA$ is still often very amenable.

\begin{Conjecture} \label{monoaciRacm}
{\rm Let $I$ be a monomial ideal of $k[x_1, \ldots, x_n]$. If $I$ is an almost complete intersection 
of finite colength its Rees algebra $\RR[I\TT]$ is almost Cohen--Macaulay.
}\end{Conjecture}

 For $d=2$ this comes from \cite{RS03}.
An important special case was settled in  \cite[Theorem 2.5]{BST15}:

\begin{Theorem} If $a_i = a$ and $b_i=b$ for all $i$, then $\RR[I\TT]$ is almost Cohen-Macaulay. 
\end{Theorem}



We will recall the artesanal method of \cite{acm} and
examine in detail the case $a=b=c=n\geq 3$ and
$\alpha=\beta = \gamma =1$. We want to argue that
$\RR[I\TT]$ is almost Cohen--Macaulay. To benefit from the monomial generators in using {\em Macaulay2} we set
$I = (xyz, x^n, y^n,z^n)$. Setting $\BB=\RR[u, \TT_1, \TT_2, \TT_3]$, we claim that
\[
\LL= (z^{n-1}u - xy\TT_3, y^{n-1}u - xz\TT_2, x^{n-1}u - yz\TT_1, z^n\TT_2-y^n\TT_3,  z^n\TT_1-x^n\TT_3, 
y^n\TT_1-x^n\TT_2, 
\]
\[ y^{n-2}z^{n-2}u^2 - x^2\TT_2\TT_3, x^{n-2} z^{n-2} u^2 -y^2 \TT_1\TT_3,   x^{n-2}y^{n-2} u^2 - z^2\TT_1\TT_2,
x^{n-3}y^{n-3}z^{n-3}u^3 - \TT_1\TT_2\TT_3). 
  \]
 We also want to show that these ideals define an almost Cohen--Macaulay Rees algebra.

\bigskip

There is a natural specialization argument. Let $X$, $Y$ and $Z$ be new indeterminates and let
$\BB_0 = \BB[X,Y,Z]$. In this ring define the ideal $\LL_0$ obtained by replacing in the list above of generators 
of $\LL$, $x^{n-3}$ by $X$ and accordingly $x^{n-2}$ by $xX$, and so on; carry  out similar actions on the other variables:

\[
\LL_0= (z^2 Zu - xy\TT_3, y^{2}Yu - xz\TT_2, x^{2}Xu - yz\TT_1, z^3Z\TT_2-y^3Y\TT_3,  z^3Z\TT_1-x^3X\TT_3, 
y^3 Y\TT_1-x^3X\TT_2, 
\]
\[ yzYZu^2 - x^2\TT_2\TT_3, x z XZ u^2 -y^2 \TT_1\TT_3,   xyXY u^2 - z^2\TT_1\TT_2,
XYZu^3 - \TT_1\TT_2\TT_3). 
  \]

Invoking {\em Macaulay2} gives a (non-minimal) projective resolution
\[ 0 \rar \BB_0^4 \stackrel{\phi_4}{\lar}
 \BB_0^{17} \stackrel{\phi_3}{\lar}
\BB_0^{22} \stackrel{\phi_2}{\lar}
 \BB_0^{10} \stackrel{\phi_1}{\lar}
  \BB_0 \lar \BB_0/\LL_0 \rar 0.
 \]

We claim that the specialization $X \rar x^{n-3}$, $Y \rar y^{n-3}$, $Z \rar z^{n-3}$ gives a projective resolution
of $\LL$.

\begin{enumerate}

\item Call $\LL'$ the result of the specialization in $\BB$. We argue that $\LL' = \LL$.

\medskip

\item Inspection of the Fitting ideal $F$ of  $\phi_4$ shows that it contains $(x^3, y^3,z^3, u^3, \TT_1\TT_2\TT_3)$. 
From standard theory, the radicals of the Fitting ideals of $\phi_2$ and $\phi_2$  contain $\LL_0$, and therefore
the radicals of the Fitting ideals of these mappings after specialization will contain the ideal $(L_1)$ of $\BB$, as 
$L_1 \subset \LL'$.

\medskip

\item Because $(L_1)$ has codimension $3$,
 by the acyclicity 
 theorem (\cite[1.4.13]{BH}) 
  the complex gives a projective resolution of $\LL'$. Furthermore, as $\mbox{\rm proj. dim }\BB/\LL' \leq 4$,
 $\LL'$ has no associated primes of codimension $\geq 5$. Meanwhile
the Fitting ideal of $\phi_4$ having codimension $\geq 5$,  forbids the existence of associated primes
of codimension $4$. Thus $\LL'$ is 
unmixed.

\medskip

\item Finally, if $(L_1) \subset \LL' $, as $\LL'$ is unmixed its associated primes are minimal primes of 
$(L_1)$, but by \cite[Proposition 2.5(iii)]{acm}, there are just two such, $\m\BB$ and $\LL$. Since
$\LL' \not \subset \m\BB$, $\LL$ is its unique associated prime. Localizing at $\LL$ gives
the equality of $\LL' $ and $\LL$ since $\LL$ is a primary component of $(L_1)$.

\end{enumerate}

Let us sum up this discussion:

\begin{Proposition} \label{nnn111} The Rees algebra of $I(n, n, n, 1, 1, 1)$, $n\geq 3$, is almost Cohen--Macaulay.
\end{Proposition}

\begin{Corollary} $\rme_1(I(n,n,n,1,1,1)) = 3(n+1)$.
\end{Corollary}

\begin{proof} Follows easily since $\rme_0(I) = 3n^2$, the colengths of the monomial ideals $I$ and $I_1(\phi)$ directly
calculated and $\red_J(I) = 2$ so that 
\[ \rme_1(I) =  \lambda(I/J) + \lambda(I^2/JI) = \lambda(I/J) + [\lambda(I/J) - \lambda(\RR/I_1(\phi))]=
(3n-1) + 4.\] 
In particular we have $s_0(I) = 4$ for the multiplicity of the Sally module $S_Q(I)$.
\end{proof}






\section{Current issues}
\noindent
We leave this as a reminder of  unfinished business.

\begin{itemize}

\item Veronese relations

\medskip

\item Reduction mod superficial elements

\medskip

\item $j$-coefficients of Sally modules

\medskip

\item $S_Q(I)$ versus $S_Q(\overline{I})$

\medskip


\end{itemize}

\section{List of definitions and notation}

\noindent
Throughout
 $(\RR, \m)$ is a Noetherian local ring of dimension $d>0$.

\begin{itemize}
\item Multiplicative 
filtration of ideals: 
a sequence of ideals $\mathcal{F} = \{I_j, j\geq 0\}$ such that $I_0 = \RR$, $I_j \subset I_{j-1}$ and
$I_j I_k \subset I_{j+k}$

\item Rees algebra of a multiplicatice
 filtration: $\BB =  \sum_{j\geq 0} I_j\TT^j= \RR\oplus \BB_{+}\subset 
\RR[\TT]$

\item Associated graded ring of $\mathcal{F}$: $\gr_{\mathcal{F}}(\RR) = \bigoplus_{j\geq 0} I_{j}/I_{j+1}$ 

\item Reductions of a Rees algebra: a Rees subalgebra $\AA=\RR[Q\TT] \subset \BB $ such that $\BB$ is finite over $\AA$

\item Reduction number:
The reduction number of $\BB$ relative to a Rees subalgebra $\AA$ is   
\[ \red(\BB/\AA) = \inf \{ n \mid \BB = \sum_{j\leq n} \AA \BB_j\} \]
\item Special fiber of a Rees algebra: $F(\BB) = \BB\otimes_{\RR} \RR/\m$

\item Sally module:
The Sally module $S_{\BB/\AA}$ of $\BB$ relative to the Rees subalgebra $\AA=\RR[Q \TT]$ is the $\AA$-module 
is defined by the exact sequence of finitely generated $\AA$--modules
 \begin{center}
\begin{eqnarray*}\label{defnSally}
0 \rar I_1\AA \lar \BB_{+}[+1]= \bigoplus_{j\geq 1} I_{j}t^{j-1} \lar S_{\BB/\AA} = \bigoplus_{j\geq 1} I_{j+1}/I_1Q^{j} \rar 0 
\end{eqnarray*}
\end{center}

\item Special fiber of a Sally module: $F(S_{\AA}(\BB)) = S_{\AA}(\BB) \otimes_{\AA} \AA/(Q\TT) = 
\bigoplus_{j=1}^{r-1} I_{j+1}/I_1 Q^j$, $r = \red(\BB/\AA)$

\end{itemize}

\end{document}